\documentclass[12pt]{amsart}
\usepackage{a4wide}
\usepackage[utf8]{inputenc}
\usepackage[T1]{fontenc} 
\usepackage[english]{babel}
\usepackage{amsmath}
\usepackage{amsfonts}
\usepackage{amssymb}
\usepackage{amsthm}
\usepackage{graphicx}
\usepackage{subcaption}
\usepackage{enumerate}
\usepackage{color}
\usepackage{todonotes}
\usepackage{bbm}
\usepackage{verbatim}
\usepackage{url}

\renewcommand{\d}{\mathrm{d}}
\newcommand{\E}[1]{\mathbb{E}\left[#1\right]} 
\newcommand{\no}[1]{\left\Vert#1\right\Vert} 
\newcommand{\nos}[1]{\left\Vert#1\right\Vert^2} 
\newcommand{\be}[1]{\left\vert#1\right\vert} 
\newcommand{\bes}[1]{\left\vert#1\right\vert^2} 

\newcommand{\fe}[1]{\frac{\nabla #1}{| \nabla #1 |_{\epsilon}}} 

\newcommand{\R}{\mathbb{R}}

\newcommand{\N}{\mathbb{N}}

\newcommand{\into}{\int\limits_{\mathcal{O}}}
\renewcommand{\epsilon}{\varepsilon}
\newcommand{\eps}{\epsilon}
\newcommand{\ska}[1]{\left( #1 \right)} 
\renewcommand{\div}{\mathrm{div}}

\newcommand{\intt}{\int_0^t}

\newcommand{\F}{\mathcal{F}}

\renewcommand{\Xi}{X_{\epsilon,h}^{i}}

\newcommand{\Hz}{\mathbb{H}^1_0}
\renewcommand{\L}{\mathbb{L}^2}
\newcommand{\Hm}{\mathbb{H}^{-1}}
\renewcommand{\phi}{\varphi}
\renewcommand{\F}{\mathcal{F}}
\renewcommand{\P}{\mathbb{P}}

\newcommand{\vh}{\varphi_h}

\renewcommand{\O}{\mathcal{O}}
\newcommand{\D}{\O}

\newcommand{\Xe}{X^{\epsilon}}

\newcommand{\Yn}{Y_{\eps,h}^{n}}
\newcommand{\Ynm}{Y_{\eps,h}^{n-1}}
\newcommand{\Ye}{Y^{\epsilon}}
\newcommand{\fn}{\Sigma_h^{n}}
\newcommand{\fnm}{\Sigma_h^{n-1}}
\newcommand{\f}{\Sigma}
\newcommand{\ftau}{\Sigma_h^{\tau}}

\newcommand{\Xn}{X_{\eps,h}^{n}}

\newcommand{\Xnm}{X_{\eps,h}^{n-1}}

\renewcommand{\F}{\mathcal{F}}
\renewcommand{\P}{\mathbb{P}}

\renewcommand{\O}{\mathcal{O}}

\newcommand{\Je}{\mathcal{J}_{\epsilon,\lambda}}
\newcommand{\J}{\mathcal{J}_{\lambda}}

\newtheorem{thms}{Theorem}[section]
\newtheorem{defs}{Definition}[section]
\newtheorem{cors}{Corollary}[section]

\newtheorem{remark}{Remark}[section]
\newtheorem{lemma}{Lemma}[section]

\newcommand{\blue}{}
\newcommand{\redd}{}
\newcommand{\red}{}
\newcommand{\xx}{x}
\begin{document} 

\title{A posteriori estimates for the stochastic total variation flow}

\author{\v{L}ubom\'{i}r Ba\v{n}as}
\address{Department of Mathematics, Bielefeld University, 33501 Bielefeld, Germany}
\email{banas@math.uni-bielefeld.de}

\author{Andr\'e Wilke}
\address{Department of Mathematics, Bielefeld University, 33501 Bielefeld, Germany}
 \email{ awilke@math.uni-bielefeld.de}

\thanks{Funded by the Deutsche Forschungsgemeinschaft (DFG, German Research Foundation) – SFB 1283/2 2021 – 317210226.}

\begin{abstract}
We derive  a posteriori error estimates for a fully discrete time-implicit finite element approximation of the stochastic total variaton flow (STVF) with additive space time noise.
The estimates are first derived for an implementable fully discrete approximation of a regularized stochastic total variation flow. We then show that the derived a posteriori estimates remain
valid for the unregularized flow up to a perturbation term that can be controlled by the regularization parameter.
Based on the derived a posteriori estimates we propose a 
pathwise algorithm for the adaptive space-time refinement
and perform numerical simulation for the regularized STVF to demonstrate the behavior of the proposed algorithm.
\end{abstract}

\maketitle
\section{Introduction}
In this paper we derive {a posteriori}  error estimates for a fully discrete finite element approximation of the  stochastic total variation flow (STVF)
 \begin{align}\label{STVF}
\d X=& \div\left(\frac{\nabla X}{\be{\nabla X}}\right) \d t -\lambda (X- g) \d t +\sigma\d W, &&\text{in } (0,T)\times \O, 
\nonumber\\
X  =& 0 && \text{on } (0,T)\times \partial \O, 
\\ \nonumber
X(0)=&x_0 &&\text{in } \O,
\end{align}
where $\O \subset \R^d$, $d\geq 1$ is a bounded, convex {polyhedral} domain, 
$\lambda \geq 0$, $T>0$ are constants and $x_0,\, g \in \Hz$ are given functions. 
{The term $W$ is a $\mathbb{R}$-valued Wiener process (Brownian motion) on a given filtered probability space $(\Omega,\F,\{\F_t\}_{0 \leq t \leq T},\P)$, 
and {$\sigma \in L^2(\Omega;{\blue C}([0,T];\Hz))$} is a $\{\mathcal{F}_t\}$-adapted stochastic process on $(\Omega,\mathcal{F},\mathbb{P})$. 
Equation (\ref{STVF}) can formally be interpreted as a stochastically perturbed gradient flow of the penalized total variation energy functional
\begin{align}\label{jfunc}
\J(u)= \into \be{\nabla u} \d \xx +\frac{\lambda}{2} \into \bes{u-g} \d \xx.
\end{align}
Due to the singular character of total variation flow (\ref{STVF}), it is convenient to perform numerical simulations using a regularized problem
\begin{align}\label{reg.STVF}
 \d \Xe&=  \div\left(\frac{\nabla \Xe}{\vert \nabla \Xe \vert_\eps}\right)\d t-\lambda(\Xe-g)\d t+\sigma\d W &&\text{in } (0,T)\times \O, \nonumber\\
 \Xe & = 0 && \text{on } (0,T)\times \partial \O,  \\
 \Xe(0)&=x_0 &&\text{in } \O \nonumber\,, 
\end{align}
where $\vert \cdot \vert_\eps= \sqrt{\vert \cdot \vert^2+\eps^2  } $ with a regularization parameter $\epsilon >0$.
 In the deterministic setting ($W\equiv 0$) equation \eqref{reg.STVF} corresponds to the gradient flow of the regularized energy functional
\begin{align}\label{def_jepslam}
\Je(u)= \into \sqrt{\bes{\nabla u} +\epsilon^2}\, \d \xx + \frac{\lambda}{2} \into \bes{u-g} \d \xx.
\end{align} 
It is well-known that the minimizers of the above regularized energy functional converge
to the minimizers of (\ref{jfunc}) for $\epsilon \rightarrow 0$, cf. \cite{Prohl_TVF_numerics} and the references therein.
Owing to the singular character of the diffusion term in (\ref{STVF})
the classical variational approach for the analysis of stochastic partial differential equations (SPDEs), see e.g. \cite{kr_07}, \cite{Roeckner_book}, is not applicable to this problem. To study singular gradient flow problems it is convenient to apply the solution framework developed in \cite{Roeckner_TVF_paper} 
which characterizes the solutions of (\ref{STVF}) as stochastic variational inequalities (SVIs). 
Convergent fully discrete finite element approximations of \eqref{STVF} and \eqref{reg.STVF} were proposed in \cite{our_paper, stvf_erratum} and \cite{stvf_martingale}. 
{Throughout the paper, we refer to the solutions which satisfy a stochastic variational inequality as SVI solutions, and to the classical SPDE solutions as variational solutions.}

We note that the use of stochastically perturbed gradient flows has proven useful in image processing.
Stochastic numerical methods for models with nonconvex energy functionals are able to avoid local energy minima and thus achieve faster convergence and/or 
more accurate results than their deterministic counterparts. We refer, for instance,
to \cite{kp06} where a stochastic level-set method is applied in image segmentation, 
and \cite{swp14} which uses stochastic gradient flow of a modified (non-convex) total variation energy functional for binary tomography.

Only few work exist on the adaptive finite element solution of the total variation problem;
we mention \cite{Veeser} and \cite{bartels2015}, \cite{bartels2021} which study a posteriori estimates and adaptive finite element schemes for the (time independent and deterministic)
minimization problems (\ref{def_jepslam}) and (\ref{jfunc}), respectively. 

{\blue 
Equation \eqref{STVF} is a special case of the $p$-Laplace flow with $p=1$.
A posteriori estimates for the (deterministic) $p$-Laplace flow for $p>1$ have been studied in \cite{cly06}, \cite{k13}, as far as we are aware the case $p=1$ has not been considered so far.
The works \cite{noch00}, \cite{mn06} considers adaptivity for the semi-discretization in time of deterministic minimal surface flows,
i.e., problem (\ref{reg.STVF}) with $\eps=1$ and $\sigma= 0$.
We also mention the works \cite{fv03}, \cite{ln05} on the a posteriori estimates for the (deterministic) mean curvature flow of graphs
which shares many similarities with the minimal surface flow equation.
}

{\blue Concerning the numerical approximation of the deterministic regularized problem (\ref{reg.STVF}) with $\sigma =0$, we mention
\cite{Prohl_TVF_numerics} which shows convergence of implicit numerical approximation of (\ref{reg.STVF}) and \cite{bdn18} which derives a priori error estimates
for a semi-implicit discretization of (\ref{reg.STVF}). 
The work \cite{cp11} proposes numerical algorithms for the deterministic counterpart of (\ref{STVF}) with $\sigma=0$ without the use of regularization
and \cite{bns14} derive a priori error estimates for the numerical approximation of the unregularized deterministic problem (\ref{STVF}).}
As far as we are aware, even in the deterministic setting,
the present work is the first to address a posteriori estimates and adaptivity for the finite element approximation of the time dependent total variation flow.
Very few results exist on adaptive finite element methods for stochastic partial differential equations (SPDEs). Algorithmic aspects of adaptive finite element approximation of SPDEs
have been addressed in \cite{sllg_book},\cite{Schellnegger}. The first work to derive rigorous a posteriori estimates for (linear) SPDEs is \cite{Prohl_adapt}; for related overview of time-adaptive
methods for stochastic differential equations we refer the reader to the references in this work. As far as we are aware, apart from the present paper,
the only other result that studies a posteriori estimates for nonlinear SPDEs is \cite{bv21}.



The present work is inspired by \cite{Prohl_adapt}, where a posteriori estimates are derived for the stochastic heat equation with additive noise.
The approach of \cite{Prohl_adapt} uses a transformation of the SPDE to a random PDE (RPDE) which enjoys improved time regularity properties that enable the postulation of an error equation in a form that can be used to derive the a posteriori estimates.
Due to the nonlinear character of the diffusion term in (\ref{reg.STVF}) the (linear) transformation used in \cite{Prohl_adapt} is not directly applicable in the present setting.
Nevertheless, we are able to adopt the approach 
by only applying the transformation to the linear part of the problem which is related to the time derivative in the resulting RPDE; an advantage of this approach is that we only
require $\mathbb{H}^1$-regularity of the noise term, as opposed to the $\mathbb{H}^2$-regularity required in \cite{Prohl_adapt}. 
We benefit from the improved time-regularity of the resulting RPDE: the fact that the solution of the RPDE possesses a generalized time derivative
enables the formulation of an error equation for the fully discrete numerical approximation.
Consequently, we derive an error bound for the finite element approximation of the regularized problem \eqref{reg.STVF}.
Furthermore, by using the convergence of \eqref{reg.STVF} to the SVI solution of \eqref{STVF} for $\varepsilon\rightarrow 0$, 
we obtain a posteriori estimates for the finite element approximation of the SVI solution of \eqref{reg.STVF}.
As far as we are aware, the present work is the first to show an a posteriori estimate of implementable numerical approximation for singular stochastic gradient flows
in the framework of stochastic variational inequalities.

{\blue
We believe that the transformation approach developed in the present paper can be adapted
to generalize a range of different results on adaptivity for deterministic nonlinear problems
(see for instance, \cite{nsv00}, \cite{ln05}, \cite{ns06}, \cite{vohr13}, \cite{vohr15}, \cite{vohr21}, and the references therein)
to the stochastic setting.
}

The paper is organized as follows. In Section~\ref{sec_notation} we introduce the notation, summarize useful analytical results on the solutions of STVF and introduce the transformation
of the SPDE (\ref{reg.STVF}) into a RPDE.
A practical fully discrete finite element scheme for the approximation of the regularized problem (\ref{reg.STVF})
is introduced in Section~\ref{Sect: Discretization}.
The a posteriori  error estimates for the proposed finite element approximation are derived in Section~\ref{sec_a_posti}.
{\blue In Section~\ref{sec_lin} we generalize the a posteriori estimate to take into account the additional error due to the linearization of the nonlinearity in the numerical scheme.}
In Section~\ref{sec_num} we formulate the adaptive finite element algorithm and perform numerical simulation
to demonstrate its properties.

\section{Notation and Preliminaries}\label{sec_notation}
For $1\leq p \leq \infty  $, we denote by $(\mathbb{L}^p,\no{\cdot}_{\mathbb{L}^p})$ the standard spaces of $p$-th order integrable functions on $\O$,
and use $\no{\cdot}= \no{\cdot}_{\L}$ and $(\cdot,\cdot)=(\cdot,\cdot)_{\L}$ for the $\L$-inner product. 
For $k \in \N$ we denote the usual Sobolev space  on $\O$ as $(\mathbb{H}^k,\no{\cdot}_{\mathbb{H}^k})$,
and $(\Hz,\no{\cdot}_{\Hz})$ stands for the $\mathbb{H}^1$ space with zero trace on $\partial \O$ with its dual space $(\Hm,\no{\cdot}_{\Hm})$. 
Furthermore, we set $\langle \cdot ,\cdot \rangle=\langle \cdot ,\cdot \rangle_{\Hm \times \Hz}$, where $\langle \cdot ,\cdot \rangle_{\Hm \times \Hz}$ 
is the duality pairing between $\Hz$ and $\Hm$. 

%
\begin{defs}\label{Bounded Variation} 
 A function $u \in \mathbb{L}^1$ is called a function of bounded variation, if its total variation
 \begin{align*}
 \into \be{\nabla u}\d\xx  = \sup\left\{-\into u\, \div\, \mathbf{v} \d\xx;~ \mathbf{v} \in C^{\infty}_0(\O,\R^d), ~\no{\mathbf{v}}_{\mathbb{L}^\infty}\leq 1\right\},
 \end{align*}
 is finite. The space of functions of bounded variations is denoted by $BV(\O)$.

For $u \in BV(\O)$ we denote
$$
 \into \sqrt{\be{\nabla u}^2 + \eps^2}\d\xx  = \sup\left\{\into \Big(-u\, \div\, \mathbf{v} + \eps\sqrt{1-|\mathbf{v}(\xx)|^2}\Big)\d\xx;~ \mathbf{v} \in C^{\infty}_0(\O,\R^d), ~\no{\mathbf{v}}_{\mathbb{L}^\infty}\leq 1\right\}\,.
$$
\end{defs}
 {We define the functionals
\begin{align*}
\bar{ \mathcal{J}}_{\eps,\lambda}(u)=
\begin{cases}
\mathcal{J}_{\eps,\lambda}(u) + \int_{\partial \O} \be{\gamma_0(u)} \mathcal{H}^{n-1} ~~ \text{for}~ u \in BV(\O)\cap \L,\\
+\infty  \qquad\qquad\qquad\qquad\quad\,\, ~~ \text{for}~ u \in BV(\O)\setminus \L
\end{cases}
\end{align*}
and 
\begin{align*}
\bar{ \mathcal{J}}_\lambda (u)=
\begin{cases}
\mathcal{J}_\lambda (u) + \int_{\partial \O} \be{\gamma_0(u)} \mathcal{H}^{n-1}  ~~\text{for}~ u \in BV(\O)\cap \L,\\
+\infty   \qquad\qquad\qquad\qquad\quad ~~\text{for}~ u \in BV(\O)\setminus \L,
\end{cases}
\end{align*} }
where $\gamma_0(u) $ is the trace of $u$ on the boundary and $\d \mathcal{H}^{n-1}$
is the Hausdorff measure on $\partial \O$.
$\bar{\mathcal{J}}_{\eps,\lambda}$ and $\bar{\mathcal{J}}_{\lambda}$ are  both convex and lower semicontinuous over $\L$ and  the lower semicontinuous hulls of $\bar{\mathcal{J}}_{\eps,\lambda}\vert_{\Hz}$ or $\bar{\mathcal{J}}_{\lambda}\vert_{\Hz}$ respectively, cf. \cite[Proposition 11.3.2]{book_attouch}.
A SVI solution of (\ref{reg.STVF}) and (\ref{STVF}) was defined in \cite{our_paper} (see also \cite{Roeckner_TVF_paper})
 as a stochastic variational inequality.
\begin{defs}\label{def_varsoleps}
Let $0  < T < \infty$, {$\varepsilon \in [0,1]$} and {$x_0 \in L^2(\Omega,\F_0;\L)$ and $g \in \L$}.
Then a $\{\F_t\}$-{adapted} map {$\Xe \in L^2(\Omega; C([0,T];\L))\cap  L^1(\Omega; L^1((0,T);BV(\O)))$ 
(denoted by $X \in L^2(\Omega; C([0,T];\L))\cap  L^1(\Omega; L^1((0,T);BV(\O)))$ for $\eps=0$)}
is called an {SVI  solution} of (\ref{reg.STVF}) (or (\ref{STVF}) if $\varepsilon=0$) if $\Xe(0)=x_0$ ($X(0)=x_0$), and
for each $\{\F_t\}$-progressively measurable process $G\in L^2(\Omega \times (0,T),\L)  $ and for each $\{\F_t\}$-adapted $\L$-valued process 
{$Z$} with $\P$-a.s. continuous sample paths, {s.t. $Z \in L^2(\Omega \times (0,T);\Hz)$}, which satisfy the equation 
\begin{align}\label{test}
\d Z(t)= -G(t) \d t +{\redd \sigma(t)}\d W(t), ~ t\in[0,T],
\end{align}
it holds for {$\eps \in (0,1]$} that
\begin{align*}
\frac{1}{2}& \E{\nos{\Xe(t)-Z(t)}}+\E{\intt {\bar{\mathcal{J}}_{\eps,\lambda}}(\Xe(s)) \d s} \nonumber\\
&\leq  \frac{1}{2} \E{\nos{x_0-Z(0)}}+\E{\intt { \bar{\mathcal{J}}_{\eps,\lambda}}(Z(s)) \d s}  \\
&
+\E{\intt \ska{\Xe(s)-Z(s),G} \d s}\,,\nonumber
\end{align*}
and analogically for $\eps=0$ it holds that
\begin{align*}
\frac{1}{2}& \E{\nos{X(t)-Z(t)}}+\E{\intt { \bar{\mathcal{J}}_{\lambda}}(X(s)) \d s} \nonumber\\
&\leq  \frac{1}{2} \E{\nos{x_0-Z(0)}}+\E{\intt { \bar{\mathcal{J}}_{\lambda}}(Z(s)) \d s}  \\
&
+\E{\intt \ska{X(s)-Z(s),G} \d s}\nonumber.
\end{align*}
\end{defs}
\begin{remark}
 In \cite{our_paper}, \cite{Roeckner_TVF_paper} SVI solutions of (\ref{STVF}) and (\ref{reg.STVF}) are defined as SVI solutions in the
sense of Definition \ref{def_varsoleps} for $\sigma \equiv X$ and $\sigma \equiv \Xe$, respectively. 
This corresponds to the multiplcative noise case, which was considered in \cite{our_paper}. 
We note that the well-posedness result from \cite{our_paper} directly caries over to the present case of (\ref{STVF}), (\ref{reg.STVF}) with additive noise.
Furthermore, we note that the results of the present paper are directly applicable to the case of $\mathbb{H}^1$-regular trace class Wiener process, cf. \cite{Prohl_adapt}.
\end{remark}
Thanks to the $\Hz$-regularity of the data $x_0$ $g$ we may follow the line of arguments in the proof of \cite[Lemma~3.1~and~3.2]{our_paper} to conclude the existence and uniqueness 
of a $\{\mathcal{F}_t\}$-adapted stochastic process $\Xe \in L^2(\Omega;C([0,T];\L))\cap L^2(\Omega;L^{\infty}([0,T];\Hz))$
for any $\eps>0$, such that the following variational formulation holds $\mathbb{P}$-a.s for all $t \in [0,T]$:
\begin{align}\label{epsilon_variational_formulation}
\ska{\Xe(t),\phi }= &\ska{x_0,\phi}-\intt\ska{\fe{\Xe(s)},\nabla\phi} \d s-\lambda\intt\ska{\Xe(s)-g,\phi} \d s
\nonumber \\
&+\intt\ska{\sigma(t),\phi} \d W(s)~~\qquad \forall \phi \in \Hz.
\end{align}
Furthermore, there exists a $C\equiv C(T)>0$ such that the following estimates holds
\begin{align}\label{l2_est}
\E{\sup\limits_{t \in [0,T]} \nos{\Xe(t)}} \leq C \left(\E{\nos{x_0}}+ \|g\|^2 +\E{\sup\limits_{t \in [0,T]}\nos{\sigma(t)}} \right)\,,
\end{align}
and
\begin{align}\label{gradient_Estimate}
\E{\sup\limits_{t \in [0,T]} \nos{\nabla \Xe(t)}} \leq C \left(\E{\nos{\nabla x_0}}+ \|\nabla g\|^2+\E{\sup\limits_{t \in [0,T]}\nos{\nabla\sigma(t)}} \right)\,.
\end{align}

{\bf Transformation into a RPDE.}
The lack of time-differentiability of the solutions of SPDEs (which is due to the low time-regularity of the driving Wiener process $W$)
is a major obstacle in the formulation of a suitable error equation for the numerical approximation
which is necessary to derive the a posteriori estimates.

To overcome this issue for the solution $\Xe$ of \eqref{reg.STVF}, we consider the transformation
$\mathbb{P}$-a.s, for a.a $t \in [0,T]$, a.e. in $\O$:
 \begin{align}\label{Transformation}
 \Ye(t)= \Xe(t)-\intt \sigma(s)\d W(s).
 \end{align}
On noting the regularity properties (\ref{l2_est}), (\ref{gradient_Estimate}) 
the triangle inequality and the Burkholder-Davis-Gundy inequality imply that
$$
\E{\sup_{t \in [0,T]}\|\Ye(t)\|_{\Hz}^2} \leq C \E{\sup_{t \in [0,T]}\|\Xe(t)\|_{\Hz}^2} + C \E{\int_0^T\|\sigma(t)\|_{\Hz}^2 \mathrm{d}t}\leq C\,,
$$
i.e., the square integrable, $\{\mathcal{F}_t\}_{0 \leq t \leq T}$-adapted process from (\ref{Transformation}) satisfies $\Ye \in L^2( \Omega;L^{\infty}([0,T];\Hz))$.
Furthermore, $\Ye$ satisfies $\mathbb{P}$-a.s for a.a. $t \in [0,T]$
\begin{align}\label{reg.RTVF}
 {\langle \partial_t \Ye(t),\phi\rangle } & =-\ska{\fe{\Xe(t)},\nabla \phi}-\lambda(\Xe(t)-g,\phi) ~~\quad \forall \phi \in \Hz\,,
\\ \nonumber
\Ye(0) & = x_0 \,.
\end{align}
By standard arguments, see for instance \cite[Theorem 5.9.3]{book_evans} (cf., also \cite[Section 2.1]{dd96}), it can be shown that
 $\partial_t \Ye \in L^2([0,T];\Hm)$, $\Ye \in C([0,T];\mathbb{L}^2)$ $\mathbb{P}$-a.s.
For instance, the $\mathbb{P}$-a.s bound $\partial_t \Ye \in L^2([0,T];\Hm)$ can be deduced (formally) from (\ref{reg.RTVF}) by the Cauchy-Schwarz and triangle inequalities
$$
\langle \partial_t \Ye,\phi\rangle \leq C \left( \left\|\fe{\Xe}\right\|+\lambda(\|\Xe\|+\|g\|)\right)\|\phi\|_{\Hz} \qquad \forall \phi \in \Hz\,,
$$
on noting (\ref{l2_est}), (\ref{gradient_Estimate}) and $\be{\fe{\cdot}} < 1$; the result can be derived rigorously by Galerkin approximation.

\begin{remark} 
In \cite{Prohl_adapt} the transformation \eqref{Transformation} is used to reformulate the considered linear SPDE into an equivalent linear RPDE which only involves
the transformed solution.
An analogous transformation of \eqref{reg.STVF} into a RPDE that only involves the transformed solution $\Ye$
is not applicable in the present setting due to the nonlinearity in \eqref{reg.STVF}.
Hence, we just transform the linear part of equation  \eqref{reg.STVF}
and, as discussed above, the existence of $\partial_t \Ye$ follows from the uniqueness and regularity which are available for the variational solution $\Xe$ of the SPDE \eqref{reg.STVF}.
We also remark that the present ''partial'' transformation has the advantage that it only requires $\mathbb{H}^1$-regularity of the noise as opposed to the $\mathbb{H}^2$-regularity required in \cite{Prohl_adapt}·
\end{remark}

\section{Numerical approximation}\label{Sect: Discretization}
{For simplicity we assume throughout this section that all discretization parameters are deterministic,
a generalization to random setting is straightforward, also, cf., Section~\ref{sec_num}.}

We consider a partition $0=t_0 < t_1< \dots< t_N=T$ of the time interval $[0,T]$
with time steps $\{\tau_n\}_{n \geq 1}$ defined as $\tau_n=t_n-t_{n-1}$. At time level $t_n$ we consider a quasi-uniform partition $\mathcal{T}_h^n$ of $\mathcal{O}$ into simplices 
which is obtained from $\mathcal{T}_h^{n-1}$ by a suitable refinement/coarsening procedure.
The set of all interior faces of the elements of the mesh $\mathcal{T}_h^n$ is denoted as
$\mathcal{E}_h^n = \bigcup_{T \in \mathcal{T}_h^n}\partial T\setminus \partial \O$. The diameter of $T \in \mathcal{T}_h^n$ and $E \in \mathcal{E}_h^n$ are denoted by $h_T$ and $h_E$, respectively. 
Given a $E \in \mathcal{E}_h^n$ we denote by $\mathcal{N}(E)$ the set of  its nodes and for $T \in \mathcal{T}_h^n, E \in \mathcal{E}_h^n$ 
we define the local patches
$\omega_T=\bigcup_{\mathcal{E}(T)\cap\mathcal{E}(T')\neq \emptyset }T'$, $\omega_E=\bigcup_{E \in \mathcal{E}(T') }T'$.
For $n=0,\dots,N$ the space $\mathbb{V}_h^n$ is the usual $\Hz$-conforming finite element space of continuous piecewise linear functions on $\mathcal{T}_h^n$. 

The $\L$-orthogonal projection $\mathcal{P}_h^n: \L\rightarrow \mathbb{V}_h^n$ is defined for $v\in\L$ as
$$
(\mathcal{P}_h^n v, \phi_h) = (v, \phi_h) \qquad \forall \phi_h \in \mathbb{V}_h^n\,.
$$
We note that the projection satisfies $\mathcal{P}_h^n v_h = v_h$ for $v_h \in \mathbb{V}_h^n$.

We also consider the Cl\'ement-Scott-Zhang interpolation operator $\Pi_n :\Hz \rightarrow \mathbb{V}_h^n$
with the following local approximation properties for $\psi \in \Hz$:
\begin{align}\label{Clement_R}
&\no{\psi -\Pi_n\psi}_T+h_T\no{\nabla[\psi-\Pi_n\psi]}_T\leq C^*h_T\no{\nabla \psi}_{\omega_T}  \qquad \forall T\in\mathcal{T}_h^n\,,
\end{align}
and 
\begin{align}\label{Clement_J}
&\no{\psi -\Pi_n \psi }_E \leq C^*h_E^{\frac{1}{2}}\no{\nabla \psi}_{\omega_E}\qquad \forall E \in \mathcal{E}_h^n\,,
\end{align}
where the constant $C^* >0$ only depends on the minimum angle of the mesh $\mathcal{T}_h^n$, see \cite{BrennerS02}.

We consider an implicit finite element approximation scheme of (\ref{reg.STVF}):
set {\redd $X^0_{\varepsilon,h} = \mathcal{P}_h^0 x_0 \in \mathbb{V}_h^0$,  $g_h = \mathcal{P}_h^0 g \in \mathbb{V}_h^0$} 
and for $n=1,\dots,N$ determine the $\mathcal{F}_{t_n}$-measurable random variable $\Xn\in \mathbb{V}_h^n$, $\mathbb{P}$-a.s., as the solution of
\begin{align}\label{num.STVF} 
(\Xn,\phi_h)=&\ska{\Xnm,\phi_h}-\tau_n\ska{\frac{\nabla \Xn}{\vert\nabla \Xn|_{\eps}},\nabla \phi_h}-\tau_n\lambda(\Xn-g_h,\phi_h)
\\ \nonumber
 &+\ska{{\redd \sigma_h}(t_{n-1})\Delta_n W,\phi_h} ~~\qquad \forall \phi_h \in \mathbb{V}_h^n,
\end{align}
where {$\Delta_n W=W(t_n)-W(t_{n-1})$}.
{\redd Furthermore, we set $\sigma_h(t_{n-1}) = \mathcal{P}^*_h \sigma(t_{n-1})$ with $\mathcal{P}^*_h: \L \rightarrow \mathbb{V}_h^{*}$;
we assume throughout the paper that $\mathbb{V}_h^{*} \subset \mathbb{V}_h^{n}$ for $n=0,\dots, N$.}

The discrete counterpart of the RPDE \eqref{reg.RTVF} is defined as follows: {\redd set $Y_{\epsilon,h}^0 = X^0_{\varepsilon,h}$}
and for $n=1,\dots, N$ determine the random variable $\Yn\in \mathbb{V}_h^{n}$ such that $\mathbb{P}$-a.s.  
\begin{align}\label{num.RTVF} 
\ska{\frac{\Yn-\Ynm}{\tau_n},\phi_h}=-\ska{\frac{\nabla \Xn}{|\nabla \Xn|_{\eps}},\nabla \phi_h}-\lambda(\Xn-g_h,\phi_h) ~~\qquad \forall \phi_h \in \mathbb{V}_h^n\,.
\end{align}

{\redd
The next lemma relates the solutions of \eqref{num.RTVF} to the solutions of the scheme \eqref{num.STVF}
and provides a discrete counterpart of the transformation (\ref{Transformation}).
\begin{lemma}\label{lem_disctrans}
Assume that $\mathbb{V}_h^{*} \subset \mathbb{V}_h^{n}$ for $n=0,\dots, N$.
Then the following discrete transformation holds for $n=1,\dots, N$
\begin{align}\label{umformen_1}
 Y^n_{\epsilon,h}&=\Xn- \sum_{i=1}^n \sigma_h(t_{i-1}) \Delta_i W.
\end{align}
\end{lemma}
\begin{proof}
We note that the solutions of \eqref{num.STVF} (and consequently of \eqref{num.RTVF}) are uniquely determined for $n=1,\dots, N$, cf. \cite{our_paper}. 
We proceed by induction. 

We consider \eqref{num.RTVF} for $n=1$ and obtain on noting $Y^0_{\epsilon,h}= X^0_{\epsilon,h}$ and \eqref{num.STVF} that
\begin{align*}
\ska{Y^1_{\epsilon,h}, \phi_h} &= \ska{Y^0_{\epsilon,h}, \phi_h} -\tau_1\ska{\frac{\nabla X^1_{\epsilon,h}}{\vert\nabla X^1_{\epsilon,h}|_{\eps}},\nabla \phi_h}-\tau_1\lambda(X^1_{\epsilon,h}-g_h,\phi_h)
\nonumber\\
&= \ska{X^0_{\epsilon,h}, \phi_h} -\tau_1\ska{\frac{\nabla X^1_{\epsilon,h}}{\vert\nabla X^1_{\epsilon,h}|_{\eps}},\nabla \phi_h} -\tau_1\lambda(X^1_{\epsilon,h}-g_h,\phi_h)
  \\ \nonumber
&\qquad   +  (\sigma_h(t_0)\Delta_1 W, \phi_h) -  (\sigma_h(t_0)\Delta_1 W, \phi_h)
\nonumber\\
&=\ska{X_{\epsilon,h}^1,\phi_h} - \ska{\sigma_h(t_0)\Delta_1 W,\phi_h}
\qquad\qquad\qquad\qquad\qquad \forall\phi_h\in \mathbb{V}_h^1\,.
\end{align*}
From the above it follows by the definition of the projection $\mathcal{P}_h^1$ that
$$
Y^1_{\epsilon,h} = \mathcal{P}_h^1\left( X_{\epsilon,h}^1-\sigma_h(t_0)\Delta_1 W \right) = X_{\epsilon,h}^1-\sigma_h(t_0)\Delta_1 W\,,
$$
where we used that $\sigma_h(t_0) = \mathcal{P}_h^*\sigma(t_0)$ and the fact that $\mathcal{P}_h^1\circ \mathcal{P}_h^* = \mathcal{P}_h^*$, since $\mathbb{V}_h^{*} \subset \mathbb{V}_h^{1}$.

Next, assuming that $Y^{n-1}_{\epsilon,h} = X_{\epsilon,h}^{n-1}- \sum_{i=1}^{n-1} \sigma_h(t_{i-1})\Delta_i W \in \mathbb{V}_h^{n-1}$ for some $n \geq 2$ we deduce similarly as above that
\begin{align*}
\ska{\Yn,\phi_h}  = & \ska{\Ynm,\phi_h} -\tau_n\ska{\frac{\nabla \Xn}{|\nabla \Xn|_{\eps}},\nabla \phi_h}-\tau_n\lambda(\Xn-g_h,\phi_h) 
\\ 
 = & \Big(X_{\epsilon,h}^{n-1}- \sum_{i=1}^{n-1} \sigma_h(t_{i-1})\Delta_i W,\phi_h\Big) -\tau_n\ska{\frac{\nabla \Xn}{|\nabla \Xn|_{\eps}},\nabla \phi_h}-\tau_n\lambda(\Xn-g_h,\phi_h)
\\
 = & \Bigg\{\ska{X_{\epsilon,h}^{n-1},\phi_h} -\tau_n\ska{\frac{\nabla \Xn}{|\nabla \Xn|_{\eps}},\nabla \phi_h}-\tau_n\lambda(\Xn-g_h,\phi_h)
\\
& + \ska{\sigma_h(t_{n-1})\Delta_n W,\phi_h}\Bigg\}- \Big(\sum_{i=1}^{n} \sigma_h(t_{i-1})\Delta_i W,\phi_h\Big)
\\
= & \Big(X_{\epsilon,h}^{n}- \sum_{i=1}^{n} \sigma_h(t_{i-1})\Delta_i W,\phi_h\Big) \qquad \qquad \forall \phi_h \in \mathbb{V}_h^n\,,
\end{align*}
which implies the statement of the lemma.
\end{proof}
}

{\blue
\begin{remark}
In case the noise is not approximated in the space $\mathbb{V}_h^*$, i.e., if we replace $\sigma_h$ by $\sigma$ in (\ref{num.STVF})
it can be shown along the proof of Lemma~\ref{lem_disctrans} that the discrete transformation becomes
\begin{align*}
{Y^n_{\epsilon,h}}&=\Xn-\mathcal{P}_h^n\left(\sum_{i=1}^n{\left(\mathcal{P}_h^{n-1}\circ \dots \circ \mathcal{P}_h^{i}\right)}  (\sigma(t_{i-1}))\Delta W_i\right).
\end{align*}
In case that there is no coarsening we have $\mathbb{V}_h^{i}\subseteq \dots \mathbb{V}_h^{n-1} \subseteq \mathbb{V}_h^n$ and the composition of the projections reduces to
$\mathcal{P}_h^n\circ \mathcal{P}_h^{n-1}\circ \dots \circ \mathcal{P}_h^{i} \equiv \mathcal{P}_h^n$ for all $i \leq n$.
In case of coarsening (i.e., if $\mathbb{V}_h^{i}\nsubseteq \mathbb{V}_h^n$ for $i < n$) it may be impractical to keep track of the resulting coarsening
error which is nonlocal in time.

If $\sigma\in C([0,T]\times \overline{\D})$ one may take $\sigma_h \equiv \mathcal{I}_h^* \sigma$ in (\ref{num.STVF})
where $\mathcal{I}_h^* : C(\overline{\D}) \rightarrow \mathbb{V}_h^*$ is the standard nodal interpolation operator.
This choice is computationally simpler than taking the projection.
\end{remark}}

To simplify the notation we denote the stochastic integral as
\begin{align}\label{stochastic_term}
\Sigma(t)=\intt \sigma(s)\d W(s)\,,
\end{align}
and its discrete counterpart as
\begin{align*}
\fn=\sum_{i=1}^n \sigma_h(t_{i-1})\Delta_i W\,.
\end{align*}
We define continuous piecewise linear time-interpolants of $\{\Yn\}_{n = 0}^{N}$,  $\{\Xn\}_{n = 0}^{N}$  and  $\{\fn\}_{n = 0}^{N}$ on $[0,T]$ as follows:
for $t \in [t_{n-1},t_n]$
we set
\begin{align}\label{Interpolation-linear_Y}
{ \redd Y^{\tau}_{\eps,h}(t)}&=\frac{t-t_{n-1}}{\tau_n}\Yn +  \frac{t_n-t}{\tau_n}\Ynm,
\end{align}
\begin{align}\label{Interpolation-linear_X}
{\redd X_{\eps,h}^{\tau}(t)}&=\frac{t-t_{n-1}}{\tau_n}\Xn + \frac{t_n-t}{\tau_n}\Xnm\,,
\end{align}
and analogically, recalling (\ref{umformen_1}),
\begin{align}\label{stoch_umschreiben}
\ftau(t):&=\frac{t-t_{n-1}}{\tau_n}\fn + \frac{t_n-t}{\tau_n}\fnm 
\\ \nonumber
&= \sum_{i=1}^{n-1}\sigma_h(t_{i-1})\Delta_i W + \frac{t-t_{n-1}}{\tau_n} \sigma_h(t_{n-1})\Delta_n W\,.
\end{align}

On noting (\ref{umformen_1}) and \eqref{Interpolation-linear_Y}, \eqref{Interpolation-linear_X}, \eqref{stoch_umschreiben}, we deduce that
\begin{align}\label{Umschreiben}
Y^{\tau}_{\eps,h}(t)&=X^{\tau}_{\eps,h}(t)-\ftau(t).
\end{align}

\section{A posteriori estimates}\label{sec_a_posti}
 
In this section we consider the numerical solution $\{\Xn\}_{n = 0}^N$ of \eqref{num.STVF} obtained using possibly non-equidistant time grid with variable time steps $\{\tau_n\}_{n = 0}^N$ 
and locally adapted spatial meshes $\{\mathcal{T}_h^n\}_{n = 0}^N$. We derive residual a posteriori estimates for the numerical solution $\{\Xn\}_{n = 0}^N$ which control
the approximation error to the variational solution $\Xe$ of \eqref{reg.STVF} as well as to the SVI solution $X$ of \eqref{STVF}. 
The constants in the derived a posteriori
estimates only depend on the given data $x_0,g, \O,T$ and the shape of the meshes $\{\mathcal{T}_h^n\}_{n = 0}^N$.


We define the (random) interior residual $\{R_h^n   \}_{n=0}^N $ as
\begin{align}\label{Residual}
R_h^n= &   \lambda (g_h-\Xn) -\frac{\Xn-\Xnm}{\tau_n} + \frac{\sigma_h(t_{n-1})\redd{\Delta_n W}}{\tau_n}.
\end{align}
The jump residual $J_E^n$ across an interior face $E = \partial K_1 \cap \partial K_2 \in \mathcal{E}_{h}^n$ is defined as
   \begin{align}\label{Jump}
   J^n_E= \frac{1}{2}\left(\frac{\nabla \Xn}{|\nabla \Xn|_{\eps}}|_{K_1}-\frac{\nabla \Xn}{|\nabla \Xn|_{\eps}}|_{K_2}\right)\cdot \nu_E,
   \end{align}
 where we use the convention that the unit normal vector $\nu_E$ to $E$ points from $K_1$ to $K_2$. 

By applying the integration by parts formula on each element $K\in \mathcal{T}_h^n$, we deduce the following equality
   \begin{align}
   \ska{\frac{\nabla \Xn}{|\nabla \Xn|_{\eps}},\nabla \phi}=  \sum_{E\in \mathcal{E}_{h}^n} \int_E J^n_E \phi \d S \qquad \forall \phi \in \Hz\,,
   \end{align}
where we used that $\nabla \cdot \frac{\nabla \Xn}{|\nabla \Xn|_{\eps}}\Big|_K = 0$, by the linearity of $\Xn|_K$.


We define the time error indicators $\eta^n_{\mathrm{time},1},\eta^n_{\mathrm{time},2}$ to be the $\R$-valued random variables
\begin{align*}
\eta^n_{\mathrm{time},1}=& \no{ \Xn -\Xnm},
\end{align*}
and
 \begin{align*}
 \eta^n_{\mathrm{time},2}=&  \no{\nabla( \Xn -\Xnm)}.
\end{align*}
The space error indicators $\eta^n_{\mathrm{space},1},\eta^n_{\mathrm{space},2}$ are defined as
\begin{align*}
\eta^n_{\mathrm{space},1}= \sum_{T\in\mathcal{T}_h^n} h_T^2\Vert R_h^n \Vert_{L^2(T)}^2\,,
\end{align*}
and
\begin{align*}
\eta^n_{\mathrm{space},2}= \sum_{E \in\mathcal{E}_{h}^n} h_E\Vert J^n_E \Vert_{L^2(E)}^2\,.
\end{align*}

We define the noise error indicators $\eta^n_{\mathrm{noise},1}$, $\eta^n_{\mathrm{noise},2}$,
{\blue
\begin{align*}
\eta^n_{\mathrm{noise},1} =& {\redd \tau_n \sum_{i=1}^{n-1}} \int_{t_{i-1}}^{t_i} \nos{\sigma(t)-\sigma(t_{i-1})}\d t
  + \tau_n \sum_{i=1}^{n-1}\tau_i\nos{  \sigma(t_{i-1}) -{\sigma_h(t_{i-1})}}
\nonumber
\\
&+
  \int_{t_{n-1}}^{t_n}\int_{t_n}^{t}\nos{\sigma(s)} \d s\, \d t + \tau^{2}_n\nos{\sigma(t_{n-1})- \sigma_h(t_{n-1})}\,,
\end{align*}
\begin{align*}
\eta^n_{\mathrm{noise},2}=& \tau_n \sum_{i=1}^{n-1}\int_{t_{i-1}}^{t_i} \nos{\nabla(\sigma(t)-\sigma(t_{i-1}))}\d t
+ \tau_n\sum_{i=1}^{n-1}\tau_i \nos{ \nabla( \sigma(t_{i-1}) -{\sigma_h(t_{i-1})})}\,,
\nonumber
\\
& + \int_{t_{n-1}}^{t_n}\int_{t_n}^{t}\nos{\nabla\sigma(s)} \d s\, \d t +   \tau^2_n\nos{\nabla (\sigma(t_{n-1}) - \sigma_h(t_{n-1}))}  \,,
\end{align*}
and 
\begin{align*}
\eta^n_{\mathrm{noise},3}=& \int_{t_{n-1}}^{t_n} \nos{\sigma(t)-\sigma(t_{n-1})}\d t
+ \tau_n \nos{\sigma(t_{n-1})-{ \sigma_h(t_{n-1})} }\,.
\end{align*}
}


\begin{thms}\label{thms_Apriori_Estimate} Let $\Xe$ be the variational solution of \eqref{epsilon_variational_formulation} and let $\{\Xn\}_{n=0}^N$ be the solution of \eqref{num.STVF} and $X^{\tau}_{\eps,h}$ the corresponding interplolation defined in \eqref{Interpolation-linear_X}. There exists a constant $C > 0$  depending $T,\be{\O},\sigma, x_0$ and  a constant $\hat{C} >0 $, which in addition depends on the shape of
the meshes $\{\mathcal{T}_h^n\}_{n = 0}^N$ and $\lambda$, such that the following a posteriori estimate holds for $1 \leq m \leq N$:
\begin{align*}
\mathbb{E}&\left[\Vert\Xe(t_m)-X^{\tau}_{\eps,h}(t_m)\Vert^2\right]
+\E{\sum_{n=1}^m \int_{t_{n-1}}^{t_n} \into \left(1-\frac{\be{\nabla {\redd X^{n}_{\eps,h}}}}{\be{\nabla {\redd X^{n}_{\eps,h}}}_\eps}\right)\frac{|\nabla (\Xe(t)-{\redd X^{n}_{\eps,h}})|^2}{ |\nabla \Xe(t)|_{\epsilon}} \d x \d t}
\\ \nonumber
&\leq
C\left(\E{\nos{{\redd x_0}-X_{\eps,h}^0}}+\lambda \nos{g-g_h}\right)+\hat{C} \Bigg\{ \sum_{n=1}^m \tau_n\E{  \left(\eta_{\mathrm{time},1}^n\right)^2+  \eta_{\mathrm{time},2}^n}
\\\nonumber
&\quad + {\blue  \E{ \sum_{n=1}^m \tau_n \eta^n_{\mathrm{space},1}}} + {\blue  \E{ \sum_{n=1}^m \tau_n\eta^n_{\mathrm{space},2}}}
+{\blue C\mathbb{E}\Bigg[ \sum_{n=1}^m \tau_n\eta_{\mathrm{space},1}^n\Bigg]^{1/2}} + {\blue C\mathbb{E}\Bigg[\sum_{n=1}^m \tau_n \eta_{\mathrm{space},2}^n\Bigg]^{1/2}}
\\
&\quad + {\blue\sum_{n=1}^m }\E{ \eta_{\mathrm{noise},1}^n
+\eta_{\mathrm{noise},2}^n+{\eta_{\mathrm{noise},3}^n}}+ 
 \left({\blue\sum_{n=1}^m} \E{\eta_{\mathrm{noise},2}^n}\right)^{\frac{1}{2}}\Bigg\}\,.
\end{align*}
\end{thms}
\begin{proof}
On noting \eqref{reg.RTVF}, \eqref{num.RTVF} and \eqref{Residual}, for any $\phi \in \Hz$, $ \phi_h \in \mathbb{V}_h^n$ and $t \in (t_{n-1},t_n)$ we deduce that $\P$-a.s.
\begin{align}\label{differnz_Eq}
\left\langle \partial_t (\Ye-Y^{\tau}_{\eps,h})(t),\phi\right\rangle
&+\ska{\frac{\nabla \Xe(t)}{|\nabla \Xe(t)|_{\eps}}-\frac{\nabla \Xn}{|\nabla \Xn|_{\eps}},\nabla \phi}
+\lambda(\Xe(t)-\Xn,\phi)-\lambda(g-g_h,\phi) \nonumber\\
&=\ska{R_h^n, \phi-\phi_h}
-\ska{\frac{\nabla \Xn}{|\nabla \Xn|_{\eps}},\nabla(\phi-\phi_h)}. 
\end{align}
We set $\phi= [Y^{\eps}-Y^{\tau}_{\eps,h}](\omega,t) \in \Hz $, $\phi_h= \Pi_n\left[Y^{\eps}-Y^{\tau}_{\eps,h}\right](\omega,t) \in \mathbb{V}_h^n$ in \eqref{differnz_Eq} and obtain after integration by parts in the respective terms, 
using the respective transformations \eqref{Transformation} and  \eqref{Umschreiben}, integrating over $(0,t_m)$ and taking the expectation that
\begin{align}\label{eq.2}
\frac{1}{2}  & \E{\nos{Y^{\eps}(t_m)-Y^{\tau}_{\eps,h}(t_m)}}
+\E{\sum_{n=1}^m \int_{t_{n-1}}^{t_n}\ska{\frac{\nabla \Xe}{|\nabla {\red \Xe}|_{\eps}}-\frac{\nabla \Xn}{|\nabla \Xn|_{\eps}},\nabla(\Xe-X^{\tau}_{\eps,h}) } \d t}
\nonumber\\
&\quad +\E{\sum_{n=1}^m \int_{t_{n-1}}^{t_n}\lambda(\Xe-\Xn,\Xe-X^{\tau}_{\eps,h})-\lambda(g-g_h,{\redd X^{\epsilon}}-X^{\tau}_{\eps,h})\d t}
\nonumber\\
& =\frac{1}{2}\E{\nos{{\redd Y^{\epsilon}(0)}-Y^{\tau}_{\eps,h}(0)}}
\nonumber\\
&\qquad+\E{\sum_{n=1}^m \int_{t_{n-1}}^{t_n}\ska{\frac{\nabla \Xe}{|\nabla \Xe|_{\eps}}-\frac{\nabla \Xn}{|\nabla \Xn|_{\eps}},\nabla(\f-\ftau) } \d t}
\nonumber\\
&\qquad+\E{\sum_{n=1}^m \int_{t_{n-1}}^{t_n}\lambda(\Xe-\Xn,\f-\ftau)-\lambda(g-g_h,\f-\ftau)\d t}
\\
 &\qquad +\E{\sum_{n=1}^m \int_{t_{n-1}}^{t_n}\ska{R_h^n, \Xe-X^{\tau}_{\eps,h}-\Pi_n\left[\Xe-X^{\tau}_{\eps,h}\right]}\d t} 
 \nonumber\\
&\qquad+\E{\sum_{n=1}^m \int_{t_{n-1}}^{t_n}\sum_{E\in \mathcal{E}_{h}^n}\int_E J^n_E\left(\Xe-X^{\tau}_{\eps,h}-\Pi_n\left[\Xe-X^{\tau}_{\eps,h}\right]\right)\d S \d t}
\nonumber\\
 &\qquad-\E{\sum_ {n=1}^m \int_{t_{n-1}}^{t_n}\ska{R_h^n, \f-\ftau-\Pi_n\left[\f-\ftau\right]}\d t }.
\nonumber 
\\ \nonumber
&\qquad-\E{\sum_{n=1}^m \int_{t_{n-1}}^{t_n}\sum_{E\in \mathcal{E}_{h}^n}\int_E J^n_E\left(\f-\ftau-\Pi_n\left[\f-\ftau\right]\right)\d S \d t}\,.
\end{align}

We rewrite the scalar product in the last sum at the left hand side  of \eqref{eq.2} as
\begin{align}\label{above}
&\lambda \ska{\Xe-\Xn, \Xe-X^{\tau}_{\eps,h}}+\lambda(g-g_h,\Xe-X^{\tau}_{\eps,h})
 \nonumber\\
&=\lambda \nos{\Xe-X^{\tau}_{\eps,h}}+\lambda \ska{X^{\tau}_{\eps,h}-\Xn, \Xe-X^{\tau}_{\eps,h}}+\lambda(g-g_h,\Xe-X^{\tau}_{\eps,h}).
\end{align}
We move the second and the third term in \eqref{above} to the right hand side of \eqref{eq.2} and estimate them follows
\begin{align*}
&-\lambda(X^{\tau}_{\eps,h}-\Xn,\Xe-X^{\tau}_{\eps,h})-\lambda(g-g_h,\Xe-X^{\tau}_{\eps,h})\\
&\leq  \frac{\lambda}{2}\nos{\Xe-\Xn}+\frac{\lambda}{2}\nos{\Xn-X^{\tau}_{\eps,h}}+4\lambda\nos{g-g_h}+\frac{\lambda}{16}\nos{\Xe-X^{\tau}_{\eps,h}}.
\end{align*}
We estimate the terms in the first sum on the right hand side in the \eqref{eq.2}. We rewrite the summands as
\begin{align}\label{Splitt_non.lin_1}
&\ska{\fe{\Xe}-\fe{\Xn},\nabla \Xe-\nabla X^{\tau}_{\eps,h}}
\nonumber\\
&=\ska{\fe{\Xe}-\fe{  \Xn},\nabla \Xe-\nabla \Xn}+\ska{\fe{\Xe}-\fe{\Xn},\nabla \Xn -\nabla X^{\tau}_{\eps,h}}
\end{align}
We estimate the {\blue first} term on the right-hand side of \eqref{Splitt_non.lin_1}  as
{\redd
\begin{align}\label{Splitt_non.lin}
& \Bigg( \fe{\Xe} -\fe{\Xn} ,\nabla {\redd \Xe} -\nabla  X^{n}_{\eps,h}\Bigg)
\nonumber
\\ \nonumber
& =\into \frac{\nabla \Xe \nabla (\Xe-\Xn)}{|\nabla \Xe|_{\epsilon}}-\frac{\nabla \Xn \nabla (\Xe-\Xn)}{|\nabla \Xe|_{\epsilon}}
\\ 
& \qquad + \frac{\nabla \Xn \nabla (\Xe-\Xn)}{|\nabla \Xe|_{\epsilon}}-\frac{\nabla \Xn \nabla (\Xe-\Xn)}{|\nabla \Xn|_{\epsilon}} \d x
\nonumber
\\
& =\into \frac{|\nabla(\Xe-\Xn) |^2}{ |\nabla \Xe|_{\epsilon}} 
+ \left(\frac{1}{|\nabla \Xe|_{\epsilon}}-\frac{1}{|\nabla \Xn|_{\epsilon}}\right) \nabla \Xn \nabla (\Xe-\Xn) \d x
\\
& \geq \into \frac{|\nabla(\Xe-\Xn) |^2}{ |\nabla \Xe|_{\epsilon}} 
- \frac{\big||\nabla \Xn|_{\epsilon}-|\nabla \Xe|_{\epsilon}\big|}{|\nabla \Xe|_{\epsilon}|\nabla \Xn|_{\epsilon}} |\nabla \Xn| |\nabla (\Xe-\Xn)| \d x
\nonumber
\\ \nonumber
& \geq \into \frac{|\nabla(\Xe-\Xn) |^2}{ |\nabla \Xe|_{\epsilon}} 
- \frac{|\nabla \Xn| }{|\nabla \Xe|_{\epsilon}|\nabla \Xn|_{\epsilon}}|\nabla (\Xe-\Xn)|^2  \d x
\\
\nonumber
& \geq \into \left( 1-  \frac{|\nabla \Xn| }{|\nabla \Xn|_{\epsilon}} \right) \frac{|\nabla(\Xe-\Xn) |^2}{ |\nabla \Xe|_{\epsilon}} \d x\,,
\end{align}
where we used $||a|_\eps-|b|_\eps|\leq |a-b|$ to obtain the last but one inequality.
}
On collecting \eqref{above}  and \eqref{Splitt_non.lin_1}, \eqref{Splitt_non.lin} we deduce from \eqref{eq.2} that
\begin{align}\label{all_terms}
\frac{1}{2}  &\E{\nos{{\redd Y^{\epsilon}(t_m)}-Y^{\tau}_{\eps,h}(t_m)}}
+\E{\sum_{n=1}^m \int_{t_{n-1}}^{t_n} \into \left(1-\frac{\be{\nabla\Xn}}{\be{\nabla\Xn}_\eps}\right)\frac{|\nabla (\Xe-\Xn)|^2}{ |\nabla \Xe|_{\epsilon}} \d x \d t}
\nonumber
\\
\nonumber
&+\E{\sum_{n=1}^m \int_{t_{n-1}}^{t_n}\lambda\nos{\Xe-\Xn}\d t}
\nonumber
\\
\nonumber
\leq &\frac{1}{2}\E{\nos{{\redd Y^{\epsilon}(0)}-Y^{\tau}_{\eps,h}(0)}}
{\redd -}\E{\sum_{n=1}^m \int_{t_{n-1}}^{t_n}\ska{\fe{\Xe}-\fe{\Xn},\nabla \Xn-\nabla X^{\tau}_{\eps,h}} \d t}
\\
&+\E{\sum_{n=1}^m \int_{t_{n-1}}^{t_n}\frac{\lambda}{2}\nos{\Xe-\Xn}\d t}
+\E{\sum_{n=1}^m \int_{t_{n-1}}^{t_n}\frac{\lambda}{2}\nos{\Xn-X^{\tau}_{\eps,h}}\d t}
\nonumber\\
&+\E{\sum_{n=1}^m \int_{t_{n-1}}^{t_n}4\lambda\nos{g-g_h}\d t}
+\E{\sum_{n=1}^m \int_{t_{n-1}}^{t_n}\frac{\lambda}{16}\nos{\Xe-X^{\tau}_{\eps,h}}\d t}
\nonumber\\
&+\E{\sum_{n=1}^m \int_{t_{n-1}}^{t_n}\ska{\frac{\nabla \Xe}{|\nabla \Xe|_{\eps}}-\frac{\nabla \Xn}{|\nabla \Xn|_{\eps}},\nabla(\f-\ftau) } \d t}
\nonumber
\\
&+\E{\sum_{n=1}^m \int_{t_{n-1}}^{t_n}\lambda(\Xe-\Xn,\f-\ftau)\d t}
+\E{\sum_{n=1}^m \int_{t_{n-1}}^{t_n}\lambda(g-g_h,\f-\ftau)\d t}
\\
 &+\E{\sum_{n=1}^m \int_{t_{n-1}}^{t_n}\ska{R_h^n, \Xe-X^{\tau}_{\eps,h}-\Pi_n\left[\Xe-X^{\tau}_{\eps,h}\right]}\d t}
  \nonumber\\
&+\E{\sum_{n=1}^m \int_{t_{n-1}}^{t_n}\sum_{E\in \mathcal{E}_{h}^n}\int_E J^n_E\left(\Xe-X^{\tau}_{\eps,h}-\Pi_n\left[\Xe-X^{\tau}_{\eps,h}\right]\right)\d S \d t}
\nonumber\\
 &-\E{\sum_{n=1}^m \int_{t_{n-1}}^{t_n}\ska{R_h^n, \f-\ftau-\Pi_n\left[\f-\ftau\right]}\d t}
  \nonumber\\
&-\E{\sum_{n=1}^m \int_{t_{n-1}}^{t_n}\sum_{E\in \mathcal{E}_{h}^n}\int_E J^n_E\left(\f-\ftau-\Pi_n\left[\f-\ftau\right]\right)\d S \d t}
\nonumber\\
=&\frac{1}{2}\nos{{\redd Y^{\epsilon}(0)}-Y^{\tau}_{\eps,h}(0)} + \sum_{1=1}^{12}A_i.
\nonumber
\end{align}
We use the bound $\be{\fe{\cdot}} < 1$ to estimate  $A_1$  with the Cauchy-Schwarz inequality 
\begin{align*}
A_1 = {\redd -}& \E{\sum_{n=1}^m \int_{t_{n-1}}^{t_n}\ska{\fe{\Xe}-\fe{\Xn},\nabla \Xn-\nabla X^{\tau}_{\eps,h}} \d t}
&\leq 4\E{\sum_{n=1}^m \tau_n  \eta_{\mathrm{time},2}^n }.
\end{align*}

{\blue The term $A_2$ can be absorbed into the corresponding term on the LHS.}

Recalling the definition \eqref{Interpolation-linear_X} of $X^{\tau}_{\eps,h}$ we deduce
{\redd
\begin{align*}
A_3& \leq \frac{\lambda}{2} \E{\sum_{n=1}^m \tau_n (\eta_{\mathrm{time},1}^n)^2}.
\end{align*}
}
Similarly, we estimate $A_5$ after applying Cauchy-Schwarz and Youngs inequality
{\redd
\begin{align*}
A_5
&\leq \E{\frac{\lambda}{4}\sum_{n=1}^m \int_{t_{n-1}}^{t_n}\nos{\Xe-\Xn}\d t+\frac{3\lambda}{8}\sum_{n=1}^m \tau_n (\eta_{\mathrm{time},1}^n)^2}.
\end{align*}
}
{
Noting the definitions $\eqref{stochastic_term}$ and $\eqref{stoch_umschreiben}$ of $\Sigma$ and $\ftau$ we deduce
{\blue
\begin{align}\label{est_sigma1}
\nonumber
& \E{\sum_{n=1}^m \int_{t_{n-1}}^{t_n}\nos{\f(t)-\ftau(t) } \d t}
\\ \nonumber 
& \leq  C\E{\sum_{n=1}^m \int_{t_{n-1}}^{t_n}\nos{\int_0^t \sigma(s)\d W(s)-\frac{t-t_{n-1}}{\tau_n}\sigma(t_{n-1})\Delta_n W-\sum_{i=1}^{n-1}\sigma(t_{i-1})\Delta_i W } \d t}
\\ 
&\quad +C\E{\sum_{n=1}^m \int_{t_{n-1}}^{t_n}\nos{\frac{t-t_{n-1}}{\tau_n}\left[ \sigma(t_{n-1}) -\sigma_h(t_{n-1})\right]\Delta_n W} \d t}
\\ \nonumber
&\quad +C\E{\sum_{n=1}^m \tau_n\nos{ \sum_{i=1}^{n-1}\sigma(t_{i-1})\Delta_i W - \sum_{i=1}^{n-1}\sigma_h(t_{i-1})\Delta_i W }}
\\ \nonumber
&=B_1+B_2+B_3.
\end{align}}
After applying  Cauchy-Schwarz and Young's inequalies  and It\^o's isometry we obtain
\begin{align}
B_1
&\leq\E{\sum_{n=1}^m \tau_n \nos{ \sum_{i=1}^{{\blue n-1}}\int_{t_{i-1}}^{t_i} \sigma(s)-\sigma(t_{i-1})\d W(s)} }
\nonumber\\
 &\quad + \E{\sum_{n=1}^m \int_{t_{n-1}}^{t_n}\nos{{\blue\int_{t_{n-1}}^{t}}\sigma(s) \d W(s) -\frac{t-{\blue t_{n-1}}}{\tau_n}\sigma(t_{n-1})\Delta_n W } \d t}
 \nonumber
\\
 &\leq  C\sum_{n=1}^m \int_{t_{n-1}}^{t_n} \E{\sum_{i=1}^{{\blue n-1}}\int_{t_{i-1}}^{t_i} \nos{\sigma(s)-\sigma(t_{i-1})}\d s}\d t
\nonumber
\\
&\quad +C\sum_{n=1}^m \int_{t_{n-1}}^{t_n}\E{{\blue\int_{t_{n-1}}^{t}}\nos{\sigma(r)} \d s }\d t 
+C\sum_{n=1}^m \tau^2_n\E{\nos{\sigma(t_{n-1})} }\,.
\nonumber
\end{align}
Similarly, we estimate $B_2,\dots, B_3$ as
{\blue
\begin{align*}
B_2
&\leq  C\E{\sum_{n=1}^m \tau^2_n\nos{  \sigma(t_{n-1}) -\sigma_h(t_{n-1})}}\,,
\end{align*}
 }
\begin{align*}
 B_3 &\leq {\blue C\E{\sum_{n=1}^m \tau_n\sum_{i=1}^{n-1}\tau_i\nos{ \sigma(t_{i-1})-\sigma_h(t_{i-1}) }}}\,.
\end{align*}
After substituting the above estimates for $B_1,\ldots,B_3$ into (\ref{est_sigma1}) we deduce
{\redd
\begin{align*}
\E{\sum_{n=1}^m \int_{t_{n-1}}^{t_n}\nos{\f(t)-\ftau(t) } \d t}\leq C\sum_{n=1}^m\E{\eta_{\mathrm{noise},1}^n}.
\end{align*}
}
Analogically one can show that
{\redd
\begin{align}\label{stoch_terms}
&\E{\sum_{n=1}^m \int_{t_{n-1}}^{t_n}\nos{\nabla(\f(t)-\ftau(t)) } \d t}
\leq C\sum_{n=1}^m\E{\eta_{\mathrm{noise},2}^n}.
\end{align}
}
Next, the Cauchy-Schwarz inequality and the bound $\be{\fe{\cdot}} < 1$ yield}
\begin{align*}
A_6  \leq & \E{\sum_{n=1}^m \int_{t_{n-1}}^{t_n}\ska{\frac{\nabla \Xe(t)}{|\nabla \Xe|_{\eps}(t)}-\frac{\nabla \Xn}{|\nabla \Xn|_{\eps}},\nabla\f(t)-\nabla \ftau(t) } \d t}
\\
\leq & 2{\blue (t_m\be{\O})^{\frac{1}{2}}} \left({\blue \sum_{n=1}^m}\E{\eta_{\mathrm{noise},2}^n}\right)^{\frac{1}{2}}\,.
\end{align*}

We estimate $A_7,A_{8}$ with the Cauchy-Schwarz and Young inequality
\begin{align*}
A_7\leq& \frac{\lambda}{4}\sum_{n=1}^m\E{\int_{t_{n-1}}^{t_n}\nos{\Xe-\Xn} \d t}+ \lambda\sum_{n=1}^m\E{\eta_{\mathrm{noise},1}^n},
\\
A_{8}\leq& \frac{\lambda}{2}T\E{\nos{g-g_h}} +\frac{\lambda}{2}\sum_{n=1}^m\E{\eta_{\mathrm{noise},1}^n}.
\end{align*} 
We use the interpolation estimate \eqref{Clement_R} and (\ref{gradient_Estimate}) to estimate
{\blue
\begin{align*}
A_9
&\leq
\E{\sum_{n=1}^m \int_{t_{n-1}}^{t_n}\sum_{T\in \mathcal{T}^n_h} \no{R_h^n}_{L^2(T)}\no{X^\eps-X^{\tau}_{\eps,h}-\Pi_n\left[X^\eps-X^{\tau}_{\eps,h}\right]}_{L^2(T)} \d t}
\\
&\leq {\blue  C^*C \mathbb{E}\Bigg[ \sup_{t \in [0,T]}\|\nabla X^\eps(t)\|   }
\sum_{n=1}^m \int_{t_{n-1}}^{t_n}\left(\sum_{T\in \mathcal{T}^n_h} h_T^2 \no{R_h^n}_{L^2(T)}^2\right)^{1/2} \d t  \Bigg]
\\
&\leq {\blue  C \mathbb{E}\Bigg[ \sup_{t \in [0,T]}\|\nabla X^\eps(t)\|   }
T^{1/2}\left(\sum_{n=1}^m \int_{t_{n-1}}^{t_n}\sum_{T\in \mathcal{T}^n_h} h_T^2 \no{R_h^n}_{L^2(T)}^2\d t\right)^{1/2} \Bigg]
\\
&\leq  {\blue  C \mathbb{E}\Bigg[ \sup_{t \in [0,T]}\|\nabla X^\eps(t)\|^2\Bigg]^{1/2}   }
\mathbb{E}\Bigg[\sum_{n=1}^m \tau_n \eta_{\mathrm{space},1}^n\Bigg]^{1/2}
\\
&\leq C \left(\E{\nos{\nabla x_0}}+ \|\nabla g\|^2+\E{\sup\limits_{t \in [0,T]}\nos{\nabla\sigma(t)}} \right)^{\frac{1}{2}}  
\mathbb{E}\Bigg[\sum_{n=1}^m \tau_n \eta_{\mathrm{space},1}^n\Bigg]^{1/2}
\\
\\
&\leq C\mathbb{E}\Bigg[\sum_{n=1}^m \tau_n \eta_{\mathrm{space},1}^n\Bigg]^{1/2}\,,
\end{align*}
}
where we also used that $\Pi_nX^{\tau}_{\eps,h} = X^{\tau}_{\eps,h}$ by the projection property of $\Pi_n$, cf., \cite{BrennerS02}.

Analogically, using (\ref{Clement_J}) we estimate $A_{10}$ as
\begin{align*}
A_{10}
&\leq {\blue C\mathbb{E}\Bigg[\sum_{n=1}^m \tau_n \eta_{\mathrm{space},2}^n\Bigg]^{1/2}}.
\end{align*}

Next, we use \eqref{Clement_R}, \eqref{Clement_J} along with \eqref{stoch_terms} and obtain after applying the Cauchy-Schwarz and Young's inequalities and It\^o's isometry 
\begin{align*}
A_{11}
&\leq 
C\E{\sum_{n=1}^m\tau_n \eta^n_{\mathrm{space},1}} + C\E{\sum_{n=1}^m \eta_{\mathrm{noise},2}^n},
\\
A_{12}
&\leq 
C\E{\sum_{n=1}^m\tau_n \eta^n_{\mathrm{space},2}} + C\E{\sum_{n=1}^m \eta_{\mathrm{noise},2}^n}.
\end{align*}
After substituting the above estimates for $A_1,\,\dots,\,A_{12}$ into \eqref{all_terms} and recalling $Y^{\epsilon}(0)=x_0$, $Y^{\tau}_{\eps,h}(0)=Y_{\eps,h}^0 =X_{\eps,h}^0$ we conclude that
\begin{align}\label{all_terms_2}
\frac{1}{2}  &\E{\nos{{\redd Y^{\epsilon}(t_m)}-Y^{\tau}_{\eps,h}(t_m)}}
+\E{\sum_{n=1}^m \int_{t_{n-1}}^{t_n} \into \left(1-\frac{\be{\nabla\Xn}}{\be{\nabla\Xn}_\eps}\right)\frac{|\nabla (\Xe(t)-\Xn)|^2}{ |\nabla \Xe(t)|_{\epsilon}} \d x \d t}
\nonumber
\\ \nonumber
\leq & C\left(\E{\nos{{\redd x_0}-X_{\eps,h}^0}}+\lambda \nos{g-g_h}\right)+\hat{C} \Bigg\{\sum_{n=1}^m \tau_n \E{  \left(\eta_{\mathrm{time},1}^n\right)^2+  \eta_{\mathrm{time},2}^n}
\\\nonumber
&\quad {\blue  \E{\sum_{n=1}^m \tau_n \eta^n_{\mathrm{space},1}}} + {\blue  \E{\sum_{n=1}^m \tau_n \eta^n_{\mathrm{space},2}}}
+{\blue C\mathbb{E}\Bigg[\sum_{n=1}^m \tau_n \eta_{\mathrm{space},1}^n\Bigg]^{1/2}} + {\blue C\mathbb{E}\Bigg[\sum_{n=1}^m \tau_n \eta_{\mathrm{space},2}^n\Bigg]^{1/2}}
\\
&\quad + {\blue\sum_{n=1}^m }\E{ \eta_{\mathrm{noise},1}^n
+\eta_{\mathrm{noise},2}^n+{\eta_{\mathrm{noise},3}^n}}+ 
 \left({\blue\sum_{n=1}^m} \E{\eta_{\mathrm{noise},2}^n}\right)^{\frac{1}{2}}\Bigg\}\,.
\end{align}
On recalling \eqref{Transformation}, \eqref{umformen_1} we obtain by the triangle and Young's inequalities
\begin{align*}
\mathbb{E}& \left[\nos{Y^\epsilon(t_m)-Y^{\tau}_{\eps,h}(t_m)}\right]
\geq&
\frac{1}{2}\E{ \nos{X^{\epsilon}(t_m)-X^{\tau}_{\eps,h}(t_m)} }-\E{ \nos{\f(t_m)-\ftau(t_m)}}\,.
\end{align*} 
Furthermore, on recalling $\eqref{stochastic_term}$, $\eqref{stoch_umschreiben}$, by the Cauchy-Schwarz inequality, Young inequality  and It\^o's isometry we get that
\begin{align*}
& \E{\nos{\f(t_m)-\ftau(t_m) } }
\\
& \qquad  \leq  C\E{\sum_{n=1}^{m}\int_{t_{n-1}}^{t_n} \nos{\sigma(t)-\sigma(t_{n-1})}\d t} + C\E{\sum_{n=1}^m \tau_n\nos{\sigma(t_{n-1})- {\sigma_h(t_{n-1})} } } 
\\
& \qquad =C\E{\sum_{n=1}^{m}\eta^n_{\mathrm{noise},3}}.
\end{align*}
Hence, the assertion of the theorem follows after substituting the above two inequalities into (\ref{all_terms_2}).
\end{proof}

{\blue
\begin{remark}
As an alternative to the estimate \eqref{Splitt_non.lin} one may use the monotonicity of the nonlinear term \cite[Lemma 3.1]{Veeser}
and obtain the control of the following error quantity
$$
\int_\O \Bigg|\frac{\big(\nabla \Xe,\eps\big)}{|\nabla \Xe|_\eps} - \frac{\big(\nabla \Xn,\eps\big)}{|\nabla \Xn|_\eps} \Bigg|^2\frac{|\Xe|_\eps + |\Xn|_\eps}{2}\d x\,,
$$
where $(\mathbf{x},\eps) = (x_1, \dots, x_d ,\eps) \in\mathbb{R}^{d+1}$ for a vector $\mathbf{x}\in \mathbb{R}^d$.
\end{remark}
}

The next lemma controls the (global) difference between the SVI solution $X$ of \eqref{STVF} and the regularized solution $\Xe$ of (\ref{reg.STVF}) in terms of the regularization parameter $\epsilon$.
\begin{lemma}\label{Lemma.eps_diff}
Let $0  < T < \infty$ and {$x_0 \in L^2(\Omega,\F_0;\Hz)$,   $ g \in \Hz$} be fixed.
Let $\Xe$ be the variational solutions of \eqref{reg.STVF} for $\epsilon\in(0,1]$.
and $X$ be the {SVI} solution  of  \eqref{STVF}. Then the following estimate holds true 
\begin{align}\label{eps_difference}
 \E{\sup\limits_{t \in [0,T]} \nos{X(t)-\Xe(t)}}\leq 2\eps |\mathcal{O}|T.
\end{align}
\end{lemma}
{\redd
\begin{proof}
We denote by $X^{\epsilon_1},X^{\epsilon_2}$ the variational solutions of \eqref{reg.STVF}
with $\epsilon\equiv\epsilon_1$, $\epsilon\equiv\epsilon_2$, respectively. 
By It\^o's formula for $\nos{X^{\epsilon_1}-X^{\epsilon_2}}$, noting that $X^{\epsilon_1}(0)-X^{\epsilon_2}(0) = 0$, we deduce $\mathbb{P}$-a.s.
\begin{align}\label{eps.difference}
\frac{1}{2}& \nos{X^{\epsilon_1}(t)-X^{\epsilon_2}(t)}
 \nonumber\\
=&-\intt \ska{\frac{\nabla X^{\epsilon_1}(s)}{ \sqrt{ \vert \nabla X^{\epsilon_1}(s) \vert^2 +\epsilon_1^2}}-\frac{\nabla X^{\epsilon_2}(s)}{ \sqrt{ \vert \nabla X^{\epsilon_2}(s) \vert^2 +\epsilon_2^2}},\nabla (X^{\epsilon_1}(s)-X^{\epsilon_2}(s))}\d s
  \\
&-\lambda\intt \nos{(X^{\epsilon_1}(s)-X^{\epsilon_2}(s)}\d s. \nonumber
\end{align}
We estimate the first term on the right-hand side of \eqref{eps.difference} using the convexity of $\sqrt{|\cdot|^2 + \eps^2}$ as
\begin{align}\label{est_convex}
 & \ska{\frac{\nabla X^{\epsilon_1}(s)}{ \sqrt{ \vert \nabla X^{\epsilon_1}(s) \vert^2 +\epsilon_1^2}}-\frac{\nabla X^{\epsilon_2}(s)}{ \sqrt{ \vert \nabla X^{\epsilon_2}(s) \vert^2 +\epsilon_2^2}},\nabla (X^{\epsilon_1}(s)-X^{\epsilon_2}(s))}
\nonumber \\
 =&\ska{\frac{\nabla X^{\epsilon_1}(s)}{ \sqrt{ \vert \nabla X^{\epsilon_1}(s) \vert^2 +\epsilon_1^2}},\nabla (X^{\epsilon_1}(s)-X^{\epsilon_2}(s))}
\nonumber \\
 &+\ska{\frac{\nabla X^{\epsilon_2}(s)}{ \sqrt{ \vert \nabla X^{\epsilon_2}(s) \vert^2 +\epsilon_2^2}},\nabla (X^{\epsilon_2}(s)-X^{\epsilon_1}(s))}
\\
 \geq& \into \sqrt{\bes{\nabla X^{\epsilon_1}}+\epsilon_1^2}-\sqrt{\bes{\nabla X^{\epsilon_2}}+\epsilon_1^2}\, \d \xx 
\nonumber\\
\nonumber
 &+\into \sqrt{\bes{\nabla X^{\epsilon_2}}+\epsilon_2^2}-\sqrt{\bes{\nabla X^{\epsilon_1}}+\epsilon_2^2}\,   \d \xx.
 \end{align}
We observe that
 \begin{align*}
\into& \left(\sqrt{\bes{\nabla X^{\epsilon_1}}+\epsilon_1^2}-\sqrt{\bes{\nabla X^{\epsilon_1}}+\epsilon_2^2}\right)  \d \xx
\\
 &=\into \frac{\left(\sqrt{\bes{\nabla X^{\epsilon_1}}+\epsilon_1^2}-\sqrt{\bes{\nabla X^{\epsilon_1}}+\epsilon_2^2}\right)\left(\sqrt{\bes{\nabla X^{\epsilon_1}}+\epsilon_1^2}+\sqrt{\bes{\nabla X^{\epsilon_1}}+\epsilon_2^2}\right)}{\sqrt{\bes{\nabla X^{\epsilon_1}}+\epsilon_1^2}+\sqrt{\bes{\nabla X^{\epsilon_1}}+\epsilon_2^2}}\d \xx
\\
 &= \into  \frac{\bes{\nabla X^{\epsilon_1}}+\epsilon_1^2-\bes{\nabla X^{\epsilon_1}}-\epsilon_2^2}{\sqrt{\bes{\nabla X^{\epsilon_1}}+\epsilon_1^2}+\sqrt{\bes{\nabla X^{\epsilon_1}}+\epsilon_2^2}}\d \xx
\\
 &= \into \frac{(\epsilon_1+\epsilon_2)(\epsilon_1-\epsilon_2)}{\sqrt{\bes{\nabla X^{\epsilon_1}}+\epsilon_1^2}+\sqrt{\bes{\nabla X^{\epsilon_1}}+\epsilon_2^2}}\d \xx
\\
 &\leq \into \be{\epsilon_1 -\epsilon_2}\left( \frac{\epsilon_1}{\sqrt{\bes{\nabla X^{\epsilon_1}}+\epsilon_1^2}}+ \frac{\epsilon_2}{\sqrt{\bes{\nabla X^{\epsilon_1}}+\epsilon_2^2}} \right)\d \xx
 \\ 
 &\leq 2\be{\O}(\epsilon_1+\epsilon_2)\,.
 \end{align*}
 Using the inequality above, we get 
 \begin{align*}
 &\into \sqrt{\bes{\nabla X^{\epsilon_1}}+\epsilon_1^2}-\sqrt{\bes{\nabla X^{\epsilon_2}}+\epsilon_2^2} \d \xx 
 +\into \sqrt{\bes{\nabla X^{\epsilon_2}}+\epsilon_2^2}-\sqrt{\bes{\nabla X^{\epsilon_1}}+\epsilon_2^2}   \d \xx\\
 &\geq - \left|\into \sqrt{\bes{\nabla X^{\epsilon_1}}+\epsilon_1^2}-\sqrt{\bes{\nabla X^{\epsilon_1}}+\epsilon_2^2} \d \xx \right|\\
 &\quad-\left|\into \sqrt{\bes{\nabla X^{\epsilon_2}}+\epsilon_1^2}-\sqrt{\bes{\nabla X^{\epsilon_2}}+\epsilon_2^2}   \d \xx\right| \\
 &\geq -2 \be{\O}(\epsilon_1+\epsilon_2).
 \end{align*}
Substituting (\ref{est_convex}) along with the last inequality into \eqref{eps.difference} yields
\begin{align}\label{Cauchy eps}
\nos{X^{\epsilon_1}(t)-X^{\epsilon_2}(t)}\leq 2 T\be{\O}(\epsilon_1+\epsilon_2).
\end{align}
Next, we take the supremum over $t \in [0,T]$ and expectation in \eqref{Cauchy eps}, set $\epsilon_2 \equiv \epsilon$ and take the limit for $\epsilon_1 \rightarrow 0$.
By the strong convergence $X^{\epsilon} \rightarrow X $ in $ L^2(\Omega; (C(0,T);\L))$, cf., \cite[Theorem 3.2]{our_paper}, we conclude that
\begin{align*}
\E{\sup\limits_{t \in [0,T]} \nos{X(t)-X^{\epsilon}(t)}}
 =& \lim_{\epsilon_1 \rightarrow  0}
\E{ \sup\limits_{t \in [0,T]} \nos{X^{\epsilon_1}(t)-X^{\epsilon}(t)}}
\leq 2 \epsilon\be{\O} T.
\end{align*}
\end{proof} 
}

The following error estimate for the numerical approximation of the SVI solution of (\ref{STVF}) is the direct consequence of Lemma~\ref{Lemma.eps_diff} and Theorem~\ref{thms_Apriori_Estimate}.
\begin{cors}\label{thms_Apriori_Estimate_2} Let $X$ be the SVI solution of \eqref{STVF}, let $\{\Xn\}_{n=0}^N$ be the solution of \eqref{num.STVF}, $X^{\tau}_{\eps,h}$ the corresponding interplolation \eqref{Interpolation-linear_X}
and let the conditions of Theorem~\ref{thms_Apriori_Estimate} hold.
Then the following estimates holds for $1 \leq m \leq N$:
\begin{align*}
&\frac{1}{2}\E{\nos{X(t_m)-X^{\tau}_{\eps,h}(t_m)}}
\\ \nonumber
&\leq
C\left(\eps + \E{\nos{{\redd x_0}-X_{\eps,h}^0}}+\lambda \nos{g-g_h}\right)+\hat{C} \Bigg\{\sum_{n=1}^m \tau_n \E{  \left(\eta_{\mathrm{time},1}^n\right)^2+  \eta_{\mathrm{time},2}^n}
\\\nonumber
&\quad + {\blue  \E{\sum_{n=1}^m \tau_n \eta^n_{\mathrm{space},1}}} + {\blue  \E{\sum_{n=1}^m \tau_n \eta^n_{\mathrm{space},2}}}
+{\blue C\mathbb{E}\Bigg[\sum_{n=1}^m \tau_n \eta_{\mathrm{space},1}^n\Bigg]^{1/2}} + {\blue C\mathbb{E}\Bigg[\sum_{n=1}^m \tau_n \eta_{\mathrm{space},2}^n\Bigg]^{1/2}}
\\
&\quad + {\blue\sum_{n=1}^m }\E{ \eta_{\mathrm{noise},1}^n
+\eta_{\mathrm{noise},2}^n+{\eta_{\mathrm{noise},3}^n}}+ 
 \left({\blue\sum_{n=1}^m} \E{\eta_{\mathrm{noise},2}^n}\right)^{\frac{1}{2}}\Bigg\}\,.
\end{align*}
\end{cors}

\section{Linearization error}\label{sec_lin}

In this section we discuss the generalization of the a posteriori estimate of Theorem~\ref{thms_Apriori_Estimate} which takes into account the linearization 
error in the (nonlinear) scheme \eqref{num.STVF}. We adopt the notation from the previous section.

We consider the following linearized scheme: given $\Xnm$ and some approximation $X_{\eps,h}^{n,*}$ determine $X_{\eps,h}^{n}$ as the solutions
of the (linear) semi-implicit scheme
\begin{align}\label{num_lin} 
(X_{\eps,h}^{n},\phi_h)=&\ska{\Xnm,\phi_h}-\tau_n\ska{\frac{\nabla X_{\eps,h}^{n}}{\vert\nabla X_{\eps,h}^{n,*}|_{\eps}},\nabla \phi_h}-\tau_n\lambda(X_{\eps,h}^{n}-g_h,\phi_h)
\\ \nonumber
 &+\ska{{\sigma_h}(t_{n-1})\Delta_n W,\phi_h} ~~\qquad \forall \phi_h \in \mathbb{V}_h^n.
\end{align}
The choice of the linearization in the above scheme corresponds to a fixed-point iteration for the approximation of the discrete nonlinear system related to (\ref{num.STVF});
i.e. $X_{\eps,h}^{n} \equiv X_{\eps,h}^{n,\ell}$, $X_{\eps,h}^{n,*} \equiv X_{\eps,h}^{n,\ell-1}$ is updated iteratively for $\ell=1,\dots$ until, e.g., $\|X_{\eps,h}^{n,\ell} - X_{\eps,h}^{n,\ell-1}\|_{\mathbb{L}^\infty}$
is below a prescribed tolerance.
The choice $X_{\eps,h}^{n,*} = X_{\eps,h}^{n-1}$ corresponds to the semi-implicit scheme which was proposed for the deterministic problem in \cite{bdn18}, the convergence of this scheme
in the stochastic setting has not been established so far.

We define the linearized analogue of (\ref{num.RTVF}) as
\begin{align}\label{lin_rpde} 
\ska{\frac{Y_{\eps,h}^{n} - \Ynm}{\tau_n},\phi_h}=-\ska{\frac{\nabla X_{\eps,h}^{n}}{|\nabla X_{\eps,h}^{n,*}|_{\eps}},\nabla \phi_h}-\lambda(X_{\eps,h}^{n}-g_h,\phi_h) ~~\qquad \forall \phi_h \in \mathbb{V}_h^n\,.
\end{align}
As in Lemma~\ref{lem_disctrans} one can show that the linear time-interpolants satisfy $Y^{\tau}_{\eps,h} = X^{\tau}_{\eps,h}-\ftau$.

Analogically to (\ref{differnz_Eq}) in the proof of Theorem~\ref{thms_Apriori_Estimate} we deduce that the solution of (\ref{num_lin}) satisfies the error equation
\begin{align}\label{errl}
\left\langle \partial_t (\Ye-Y^{\tau}_{\eps,h})(t),\phi\right\rangle
+\ska{\frac{\nabla \Xe(t)}{|\nabla \Xe(t)|_{\eps}}-\frac{\nabla \Xn}{|\nabla \Xn|_{\eps}},\nabla \phi}
+\lambda(\Xe(t)-\Xn,\phi)-\lambda(g-g_h,\phi) \nonumber\\
\qquad =\ska{R_h^{n}, \phi-\phi_h}
-\ska{\frac{\nabla \Xn}{|\nabla \Xn|_{\eps}},\nabla(\phi-\phi_h)} + \ska{\frac{\nabla \Xn}{|\nabla X_{\eps,h}^{n,*}|_{\eps}} - \frac{\nabla \Xn}{|\nabla \Xn|_{\eps}},\nabla\phi}. 
\end{align}
As in the proof of Theorem~\ref{thms_Apriori_Estimate} we set $\phi= Y^{\eps}-Y^{\tau}_{\eps,h}$, $\phi_h= \Pi_n\left[Y^{\eps}-Y^{\tau}_{\eps,h}\right]$ and integrate the result over $(0,t_m)$.
We only need to analyze the last term on the right-hand side of (\ref{errl}), the remaining terms can be estimated exactly as in Theorem~\ref{thms_Apriori_Estimate}.

We get using the Cauchy-Schwarz and Young's inequalities
\begin{align}\label{lin_est}
\nonumber
& \sum_{n=1}^m
\int_{t_{n-1}}^{t_n}\ska{\frac{\nabla \Xn}{|\nabla X_{\eps,h}^{n,*}|_{\eps}} - \frac{\nabla \Xn}{|\nabla \Xn|_{\eps}},\nabla(Y^{\eps}-Y^{\tau}_{\eps,h})}
\\
& \qquad \leq 
\frac{1}{2}\sum_{n=1}^m \tau_n \left\| \frac{\nabla \Xn}{|\nabla X_{\eps,h}^{n,*}|_{\eps}} - \frac{\nabla \Xn}{|\nabla \Xn|_{\eps}} \right\|^2 
+ \frac{1}{2} \sum_{n=1}^m \int_{t_{n-1}}^{t_n} \|\nabla(Y^{\eps}-Y^{\tau}_{\eps,h})\|^2\,.
\end{align}
On recalling that $Y^{\eps}= X^\eps-\f$, $Y^{\tau}_{\eps,h} = X^{\tau}_{\eps,h}-\ftau$ and noting the definition of $\eta_{\mathrm{noise},2}^n$, $\eta_{\mathrm{time},2}^n$
we estimate
$$
\int_{t_{n-1}}^{t_n} \E{\|\nabla(Y^{\eps}-Y^{\tau}_{\eps,h})\|^2}\leq C\tau_n \E{(\eta_{\mathrm{time},2}^n)^2 + \eta_{\mathrm{noise},2}^n}\,.
$$
Finally, the first term on the right-hand side in (\ref{lin_est}) yields the linearization error indicator
$$
\eta_{\mathrm{lin}}^n = \E{\left\| \frac{\nabla \Xn}{|\nabla X_{\eps,h}^{n,*}|_{\eps}} - \frac{\nabla \Xn}{|\nabla \Xn|_{\eps}} \right\|^2}\,.
$$
\section{Numerical experiments}\label{sec_num}

We perform the numerical experiments using the fully discrete finite element scheme (\ref{num.STVF})
on the unit square $\D= (0,1)^2$ with a slightly more general noise term. 
The scheme for $n=1, \dots, N$ then reads as
\begin{align}\label{fem_scheme}
\ska{X^n_{\varepsilon,h},\vh} =& \ska{X^{n-1}_{\varepsilon,h},\vh}-\tau_n \ska{\fe{X^n_{\varepsilon,h}},\nabla\vh } 
\nonumber \\ & 
-\tau_n\lambda\ska{X^n_{\varepsilon,h} - g_h^n,\vh}+ \widetilde{\sigma} \ska{ {\Delta_n \mathbf{W}},\vh} &&\forall \vh \in \mathbb{V}_h^n\,, 
\end{align}
where $g_h^n\in\mathbb{V}_h^n$ is a suitable approximations of the function $g$ and $\widetilde{\sigma}$ is a constant.

Furthermore, we set $X^0_{\varepsilon,h} = 0$ (i.e., we use a homogeneous initial condition $x_0\equiv 0$)
and $T=0.05$, $\lambda=200$ in all experiments.

The initial time step is chosen as $\tau_0 = 10^{-5}$, 
and $\eps=2^{-5}$, $\widetilde{\sigma}=0.25$,  if not mentioned otherwise.
{\blue The finite element space $\mathbb{V}_h^*$ is subordinate to the ''macro'' triangulation $\mathcal{T}_h^*$ which is constructed by dividing the domain $\D$ into squares with side $h = 2^{-5}$,
each square is sub-divided into four triangles with equal size; we set $\mathbb{V}_h^0\equiv \mathbb{V}_h^*(\mathcal{T}_h^*)$.
Below, we denote by $\mathcal{I}_h^n: C(\overline{\O})\rightarrow \mathbb{V}_h^n$ the standard nodal interpolation operator on $\mathbb{V}_h^n$.}

The space-time noise in (\ref{fem_scheme}) has the form
$$
{\blue \Delta_n \mathbf{W}(x,y) = \mathcal{I}_h^0\big(\sin(4\pi x) \sin(4\pi y)\Big) \Delta_n \beta_1 + \mathcal{I}_h^0\big(\sin(5\pi x) \sin(5\pi y)\Big) \Delta_n \beta_2\,,}
$$
where ${\beta}_1$, $\beta_2$ are independent scalar-valued Wiener processes; {\blue we employ the nodal interpolation instead of the $\mathbb{L}^2$-projection $\mathcal{P}_h^0$ for simplicity}.

We take the function $g$ to be the characteristic function of a circle with radius $0.25$, 
and set $g_h^n = \mathcal{I}_h^n g + {\blue \xi^{*}_h} \in \mathbb{V}_h^n$ 
with $\displaystyle \xi_h^*(x) = 0.1 \sum_{\ell=1}^{L} \phi_\ell (x) \xi_\ell$, $x\in\D$
where $\xi_\ell$, $\ell=1,\dots, L$ are realizations of independent $\mathcal{U}(-1,1)$-distributed random variables
and $\{\phi_\ell\}_{\ell=1}^L$ is the standard Lagrange basis of the space $\mathbb{V}_h^0 \equiv \mathbb{V}_h^*$.
The corresponding realization of $\xi_h^*$ is displayed in Figure~\ref{fig_data} (right).
\begin{figure}[!htp]
\center
\includegraphics[width=0.3\textwidth]{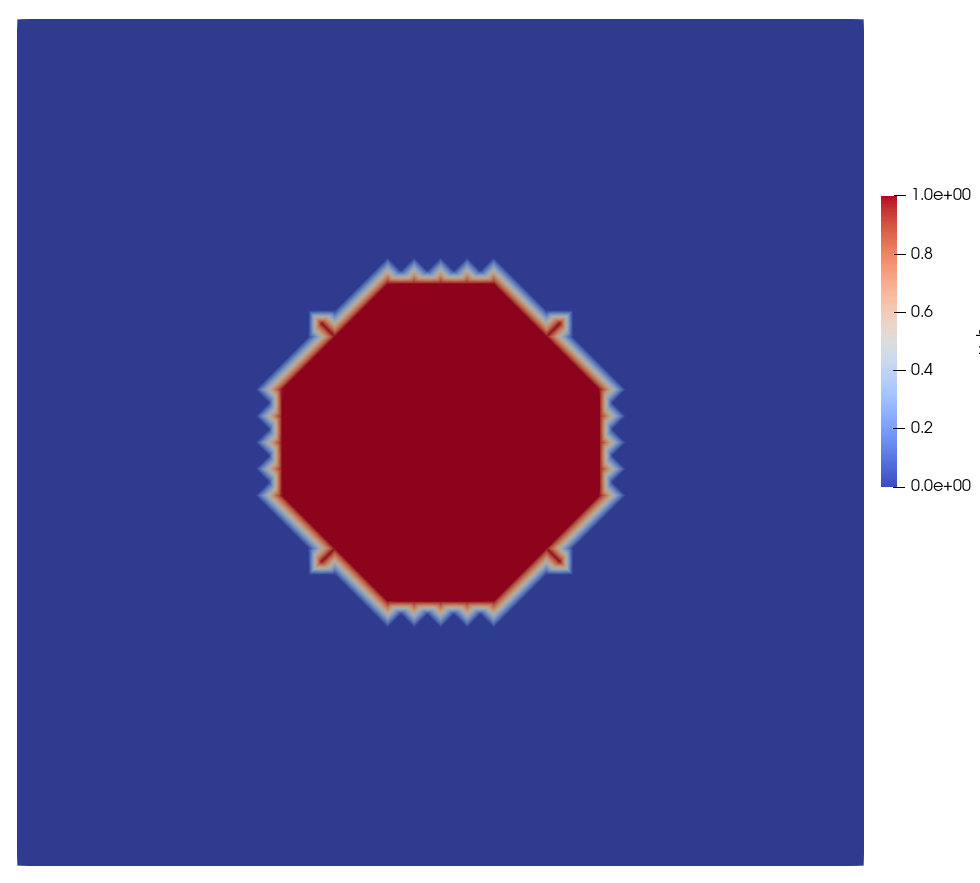}
\includegraphics[width=0.3\textwidth]{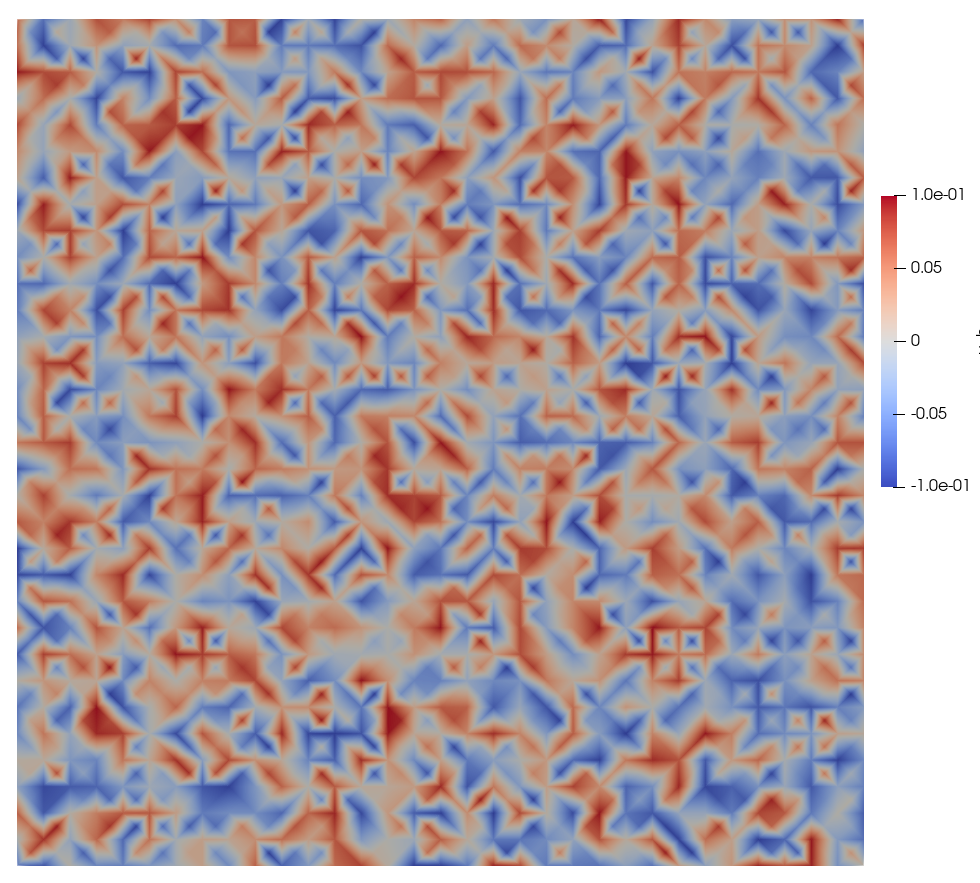}
\includegraphics[width=0.3\textwidth]{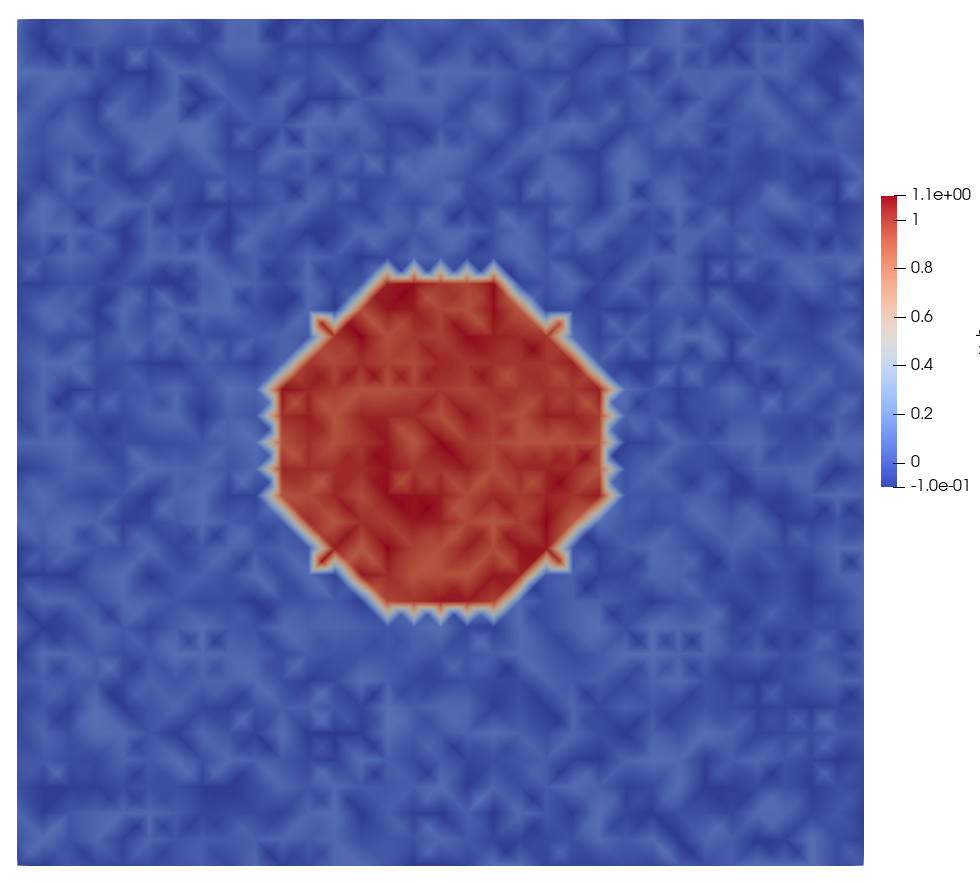}
\caption{The function $g_h^0$ (left), the noise $\xi_h^0$ (middle) and $g_h^0+\xi_h^0$ (right).}
\label{fig_data}
\end{figure}

For simplicity we employ a pathwise refinement algorithm which is explicit in time. That is, we use a Monte-Carlo approach and
for each $\omega\in\Omega$ we determine $\tau_n\equiv \tau_n(\omega)$ and $\mathcal{T}_h^n \equiv \mathcal{T}_h^n(\omega)$ using the solution from the previous time step as follows.
Given a realization $X^{n-1}_{\varepsilon,h}(\omega_k)$ of the random variable $X^{n-1}_{\varepsilon,h}$
and the tolerances $TOL_\tau$, $TOL_h$ we compute the realization $X^{n}_{\varepsilon,h}(\omega_k)$ using the scheme (\ref{fem_scheme}) with $\tau_n\equiv \tau_n(\omega_k)$, $\mathbb{V}_h^n\equiv \mathbb{V}_h^n(\mathcal{T}_h^n(\omega_k))$
The triangulation $\mathcal{T}_h^{n}(\omega_k)$ is obtained from $\mathcal{T}_h^{n-1}(\omega_k)$ by local refinement an coarsening using an error equidistribution strategy, cf., \cite{ej91}:
\begin{itemize}
\item[(i)] For all $T\in \mathcal{T}_h^{n-1}(\omega_k)$ compute the local error estimate 
$$
\displaystyle \eta^n_T = \eta_{\mathrm{space},1}^{n-1}(T) + \sum_{E\in\partial\overline{T}}\eta_{\mathrm{space},2}^{n-1}(E)\,,
$$
using the values of $X^{n-1}_{\varepsilon,h}(\omega_k)$,
\begin{itemize}
\item[(a)] if $\eta^n_T > 0.9\, TOL_h/ (\# \mathcal{T}_h^{n-1})^{1/2}$ mark $T$ for refinement,
\item[(b)] if $\eta^n_T < 0.1\, TOL_h /(\# \mathcal{T}_h^{n-1})^{1/2}$ mark $T$ for coarsening;
\end{itemize}
\item[(ii)] construct $\mathcal{T}_h^{n}(\omega_k)$ by local refinement/coarsening of elements in $\mathcal{T}_h^{n-1}(\omega_k)$,
\end{itemize}
where $\# \mathcal{T}^n_h$ denotes the number of nodes (degrees of freedom) of the mesh $\mathcal{T}^n_h$.

Once the triangulation $\mathcal{T}_h^{n}(\omega_k)$ has been constructed we compute $X^{n}_{\varepsilon,h}(\omega_k)\in \mathbb{V}_h^n(\mathcal{T}_h^n(\omega_k))$
using a simple fixed-point nonlinear solver (cf. \eqref{num_lin}).
The fixed-point iteration is terminated once the difference of two subsequent iterates in the $\mathbb{L}^\infty$-norm drop below the tolerance $10^{-4}$;
with this tolerance the value of the linearization error indicator $\eta_{\mathrm{lin}}^n$ was roughly two orders of magnitude less than the values of the time error indicator $\eta^{n}_{\mathrm{time},2}$.

Given (the approximation of) $X^{n}_{\varepsilon,h}(\omega_k)$ we determine the new time step $\tau_{n+1}(\omega_k)$ as follows
\begin{itemize}
\item[(i)] Compute the time-error indicator $\eta^{n}_{\mathrm{time},2}(\omega_k)$;
\begin{itemize}
\item[(a)] if $\eta^{n}_{\mathrm{time},2}(\omega_k) > TOL_{\tau}$ (or the fixed point iterations require $>30$ iterations)  set $\tau^{n+1}(\omega_k) = 0.5\tau^{n}(\omega_k)$;
\item[(b)] if $\eta^{n}_{\mathrm{time},2}(\omega_k) < 0.3 TOL_{\tau}$ (and the fixed point iterations require $<15$ iterations) set $\tau^{n+1}(\omega_k) = 1.5\tau^{n}(\omega_k)$;
\end{itemize}
\item[(ii)] set $n\leftarrow n+1$ and proceed to next time level.
\end{itemize}

{\blue 
We note that due to the fact that the above algorithm is explicit (i.e., we do not iterate the space-time refinement loop until the tolerance is met),
the value of the error indicators $\eta^{n}_{\mathrm{time},2}$, $\eta^n_h = \sum_{T\in \mathcal{T}_h^n} \eta^n_T$ may overshoot the respective tolerance $TOL_{\tau}$, $TOL_{h}$.
Nevertheless, the numerical results below indicate that, except for an overshoot in the initial stages of the computations,
the time averages of the error indicators remain below the prescribed tolerances.

In the present setting the noise function $\sigma\equiv\sigma(\omega,t,x,y)$ takes the form $\sigma \equiv \sigma(x,y) \equiv \widetilde{\sigma}\big(\sin(4\pi x) \sin(4\pi y) + \sin(5\pi x) \sin(5\pi y)\big)$.
Hence, by the approximation properties of the nodal interpolation operator we deduce that $\eta_{\mathrm{noise},1}^n \approx \eta_{\mathrm{noise},3}^n \approx \tau_n h_*^2 \|\sigma \|_{\mathbb{H}^2}^2$, 
$\eta_{\mathrm{noise},2}^n \approx \tau_n h_* \|\sigma \|_{\mathbb{H}^2}^2$, where $h_*$ is the mesh size of the macro triangulation $\mathcal{T}_h^*$.
Therefore, the contribution of the noise error indicators can be estimated a priori and it is not necessary to include them in the adaptive algorithm.
}

We perform the computations below with the tolerance $(TOL_h, TOL_\tau) = TOL_k$ $k=0,1,2$ where $TOL_k = 2^{-k}(2,0.25)$.
In Figure~\ref{fig_sol_tol2} we display one path of the solution computed by the adaptive algorithm with the tolerance $TOL_0$,
the corresponding meshes are displayed in Figure~\ref{fig_mesh_tol2}.
We observe that initial the mesh is refined within the circle
due to the effect of the penalization term $\lambda$ which dominates the error estimate
in the early stages of the computation, since we start with an initial condition that is far from the function $g$.
In the later stages the mesh refinement concentrates along the edge of the circle with radius $0.25$, i.e., along the discontinuity of the function $g$, cf., \cite{Veeser}, \cite{bartels2021}.

\begin{figure}[!htp]
\center
\includegraphics[width=0.24\textwidth]{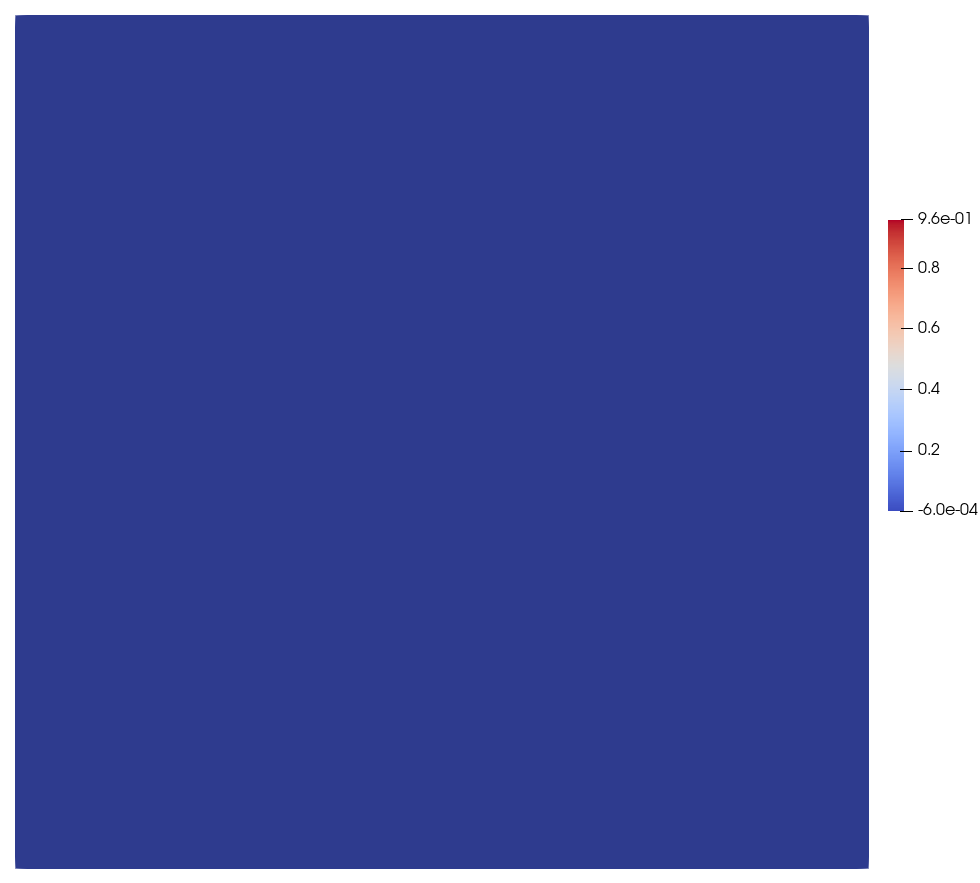}
\includegraphics[width=0.24\textwidth]{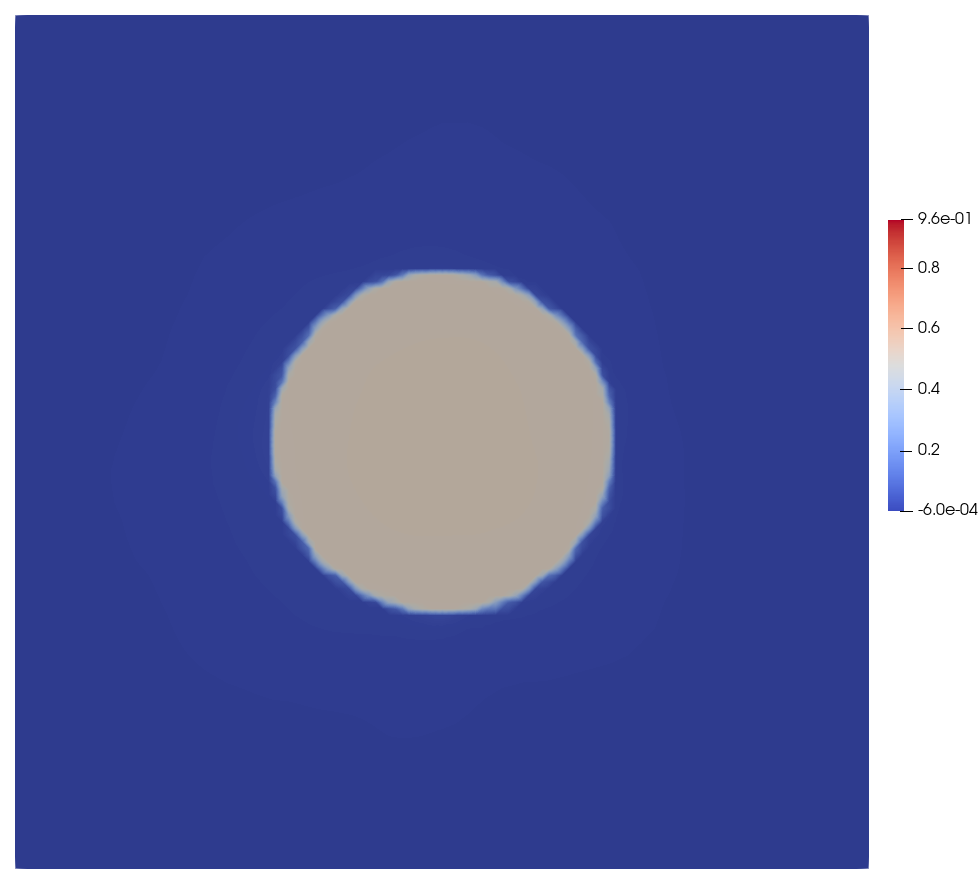}
\includegraphics[width=0.24\textwidth]{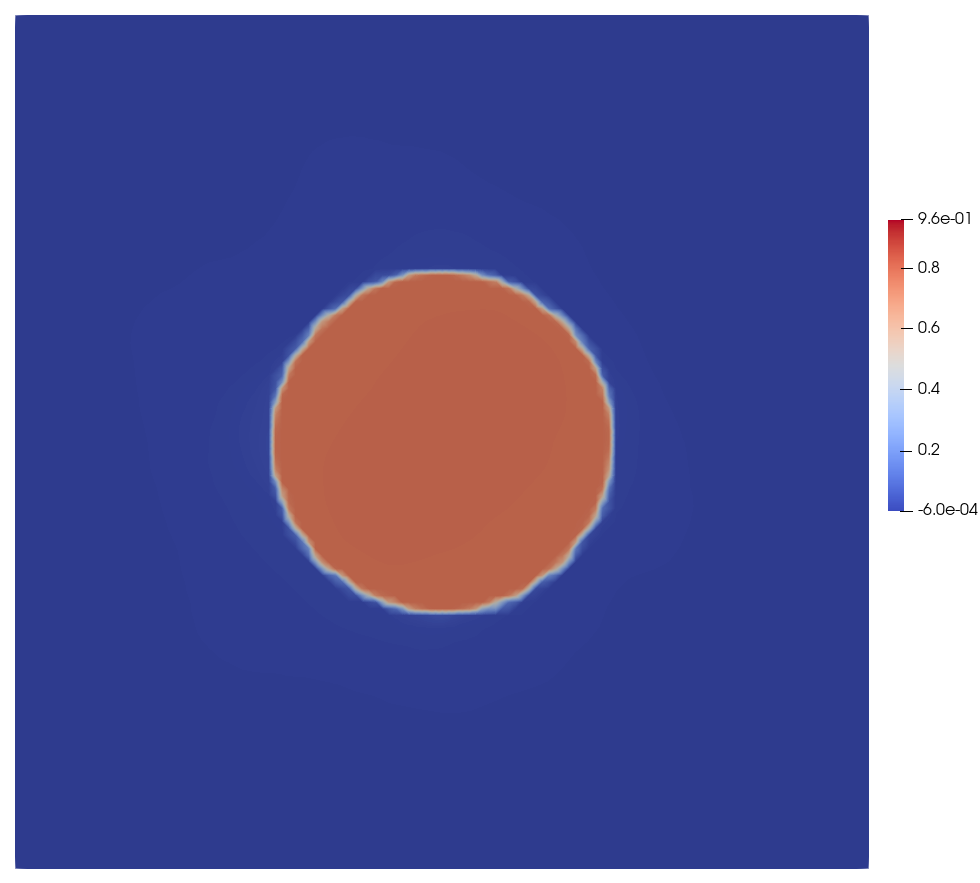}
\includegraphics[width=0.24\textwidth]{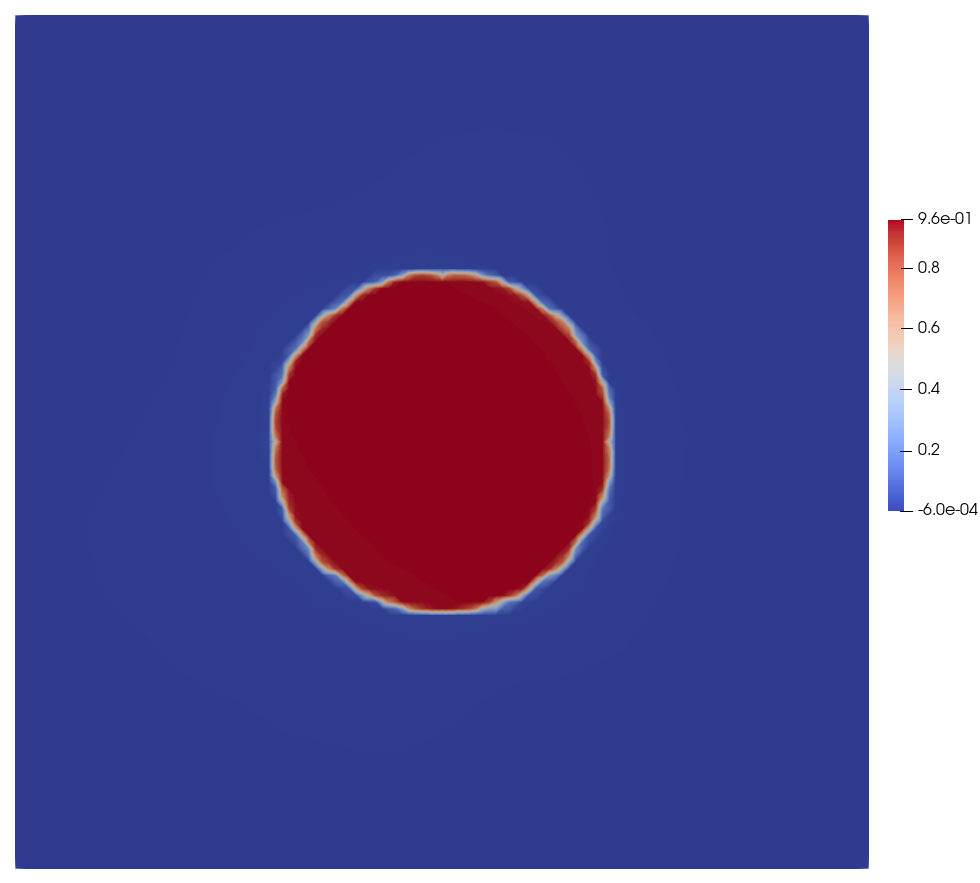}
\caption{Solution at time $t=0,0.0017,0.0026,0.05$ computed with tolerance $TOL_0$.}
\label{fig_sol_tol2}
\end{figure}
\begin{figure}[!htp]
\center
\includegraphics[width=0.24\textwidth]{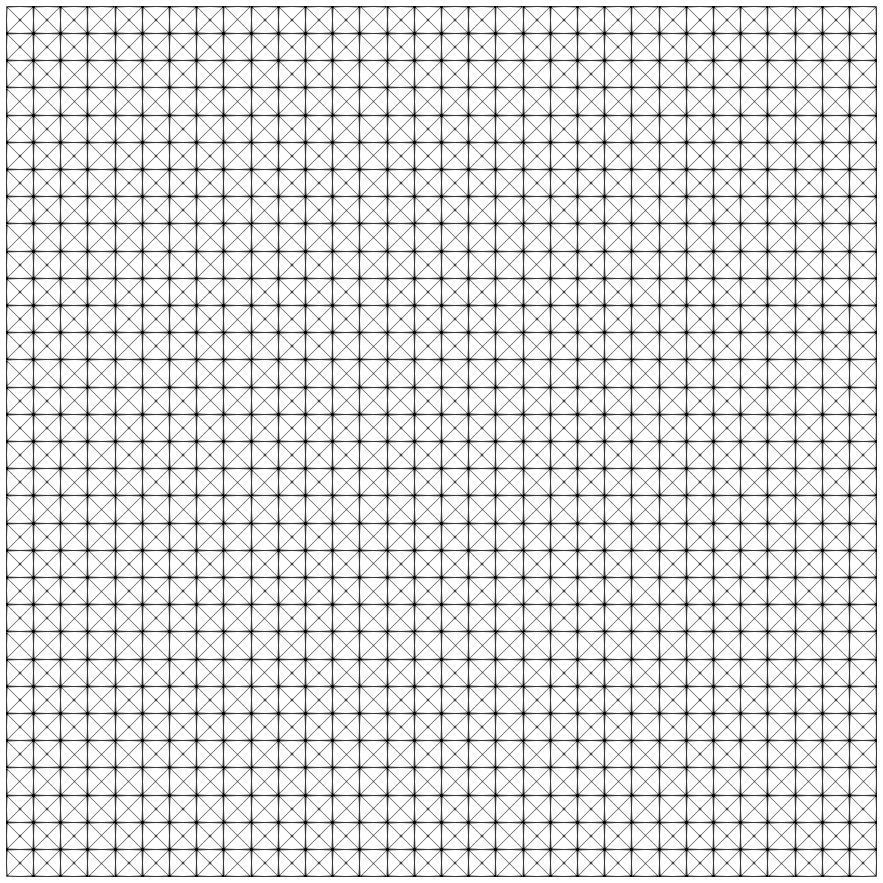}
\includegraphics[width=0.24\textwidth]{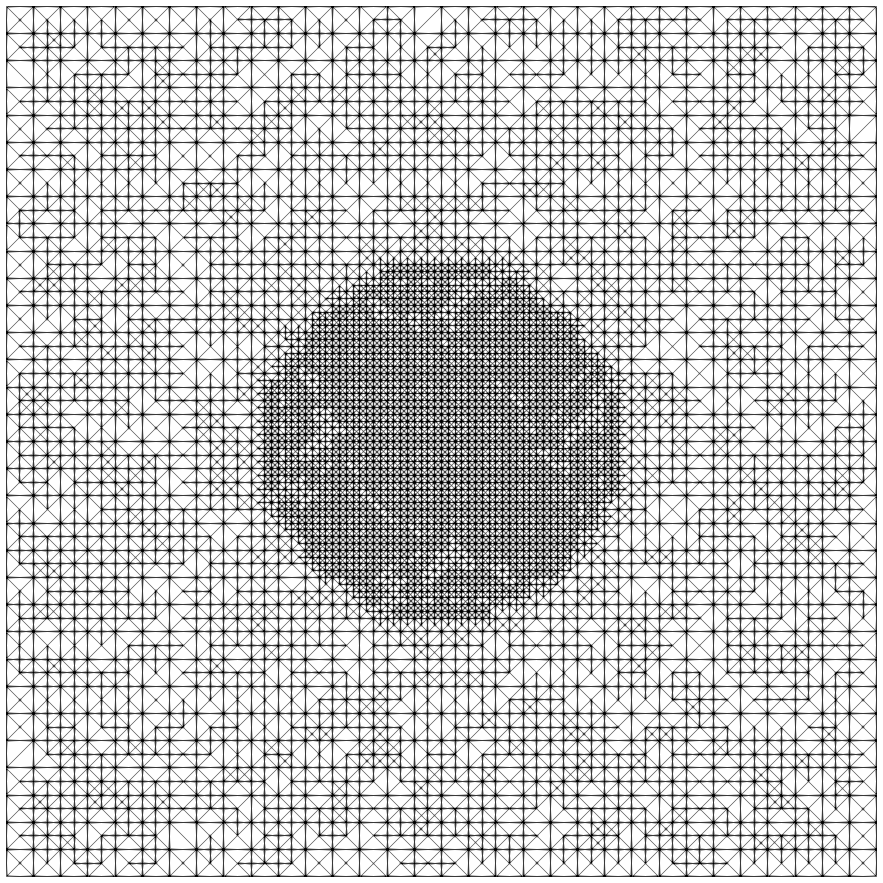}
\includegraphics[width=0.24\textwidth]{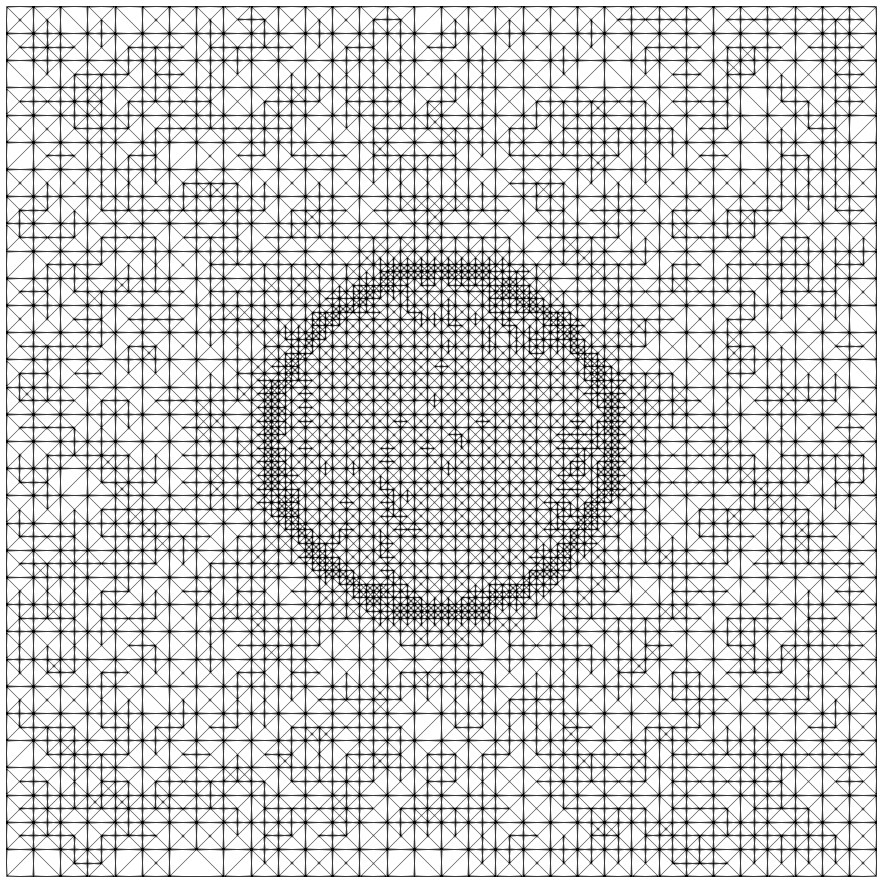}
\includegraphics[width=0.24\textwidth]{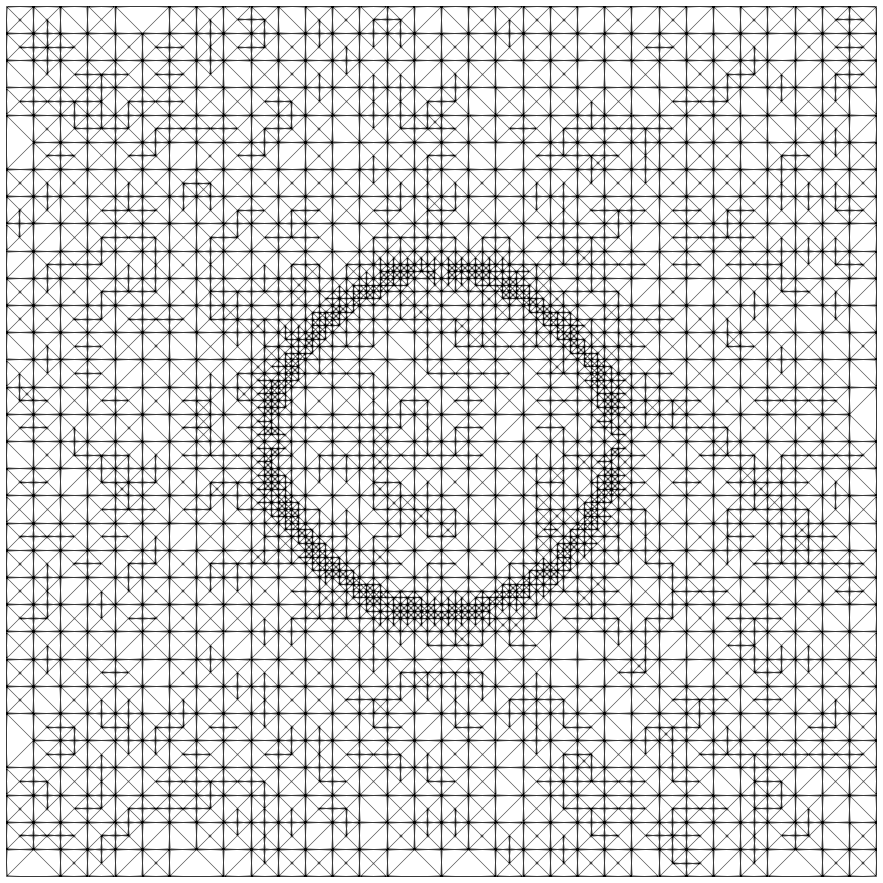}
\caption{Finite element mesh at time $t=0,0.0017,0.0026,0.05$ computed with tolerance $TOL_0$.}
\label{fig_mesh_tol2}
\end{figure}

One path of the solution computed by the adaptive finite element algorithm with finer tolerance $TOL_1$ is displayed in Figure~\ref{fig_sol_tol1}, 
the corresponding meshes are displayed in Figure~\ref{fig_mesh_tol1}.
The adaptive algorithm behaves analogically as in the previous experiment. Eventually, the meshes are refined along the discontinuity of the solution,
with a mesh size that is four times smaller than in the case of the larger tolerance $TOL_0$.
\begin{figure}[!htp]
\center
\includegraphics[width=0.24\textwidth]{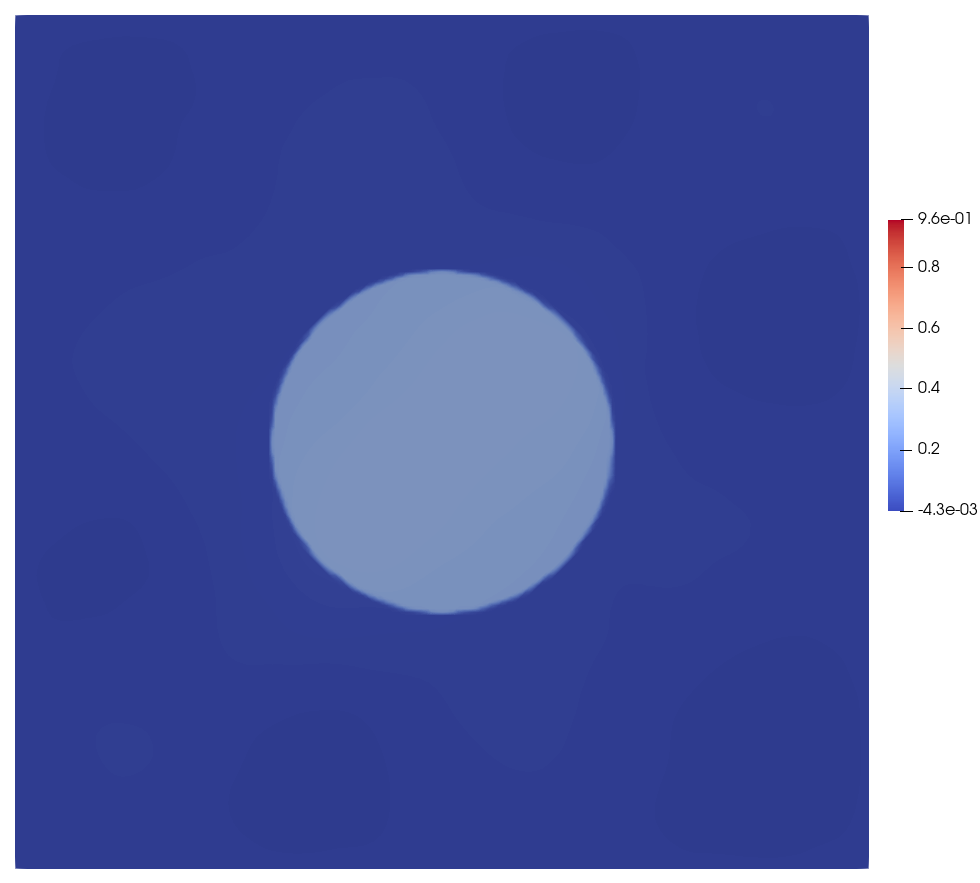}
\includegraphics[width=0.24\textwidth]{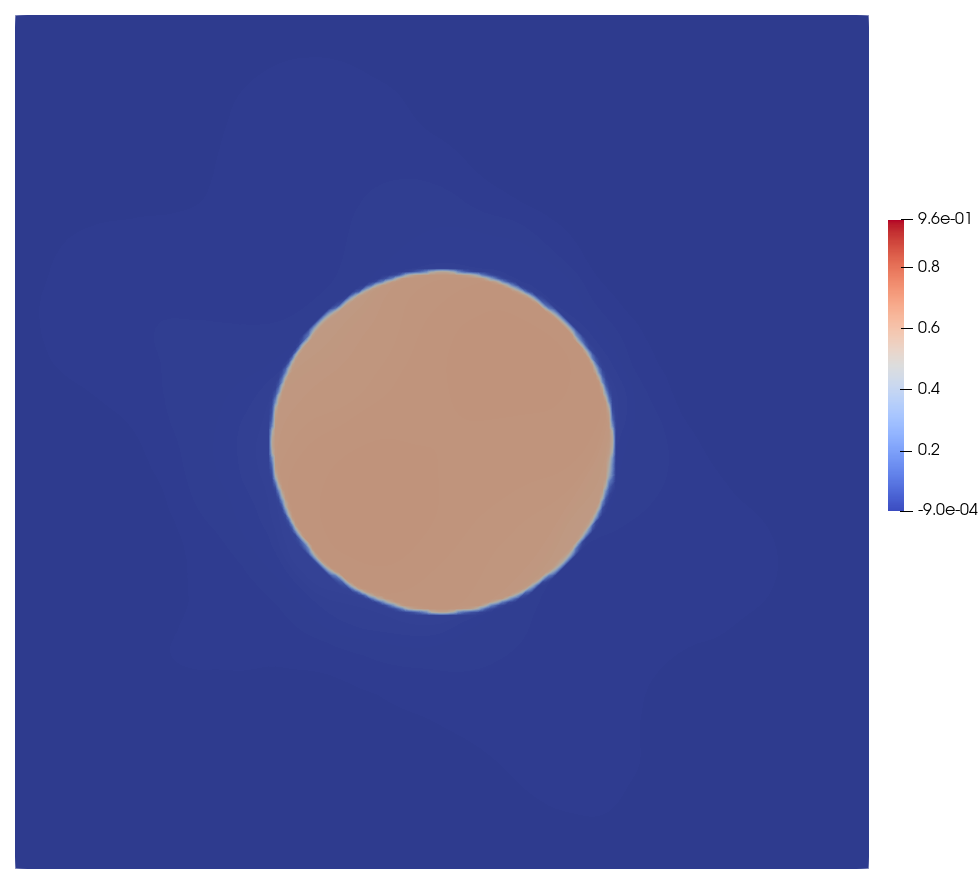}
\includegraphics[width=0.24\textwidth]{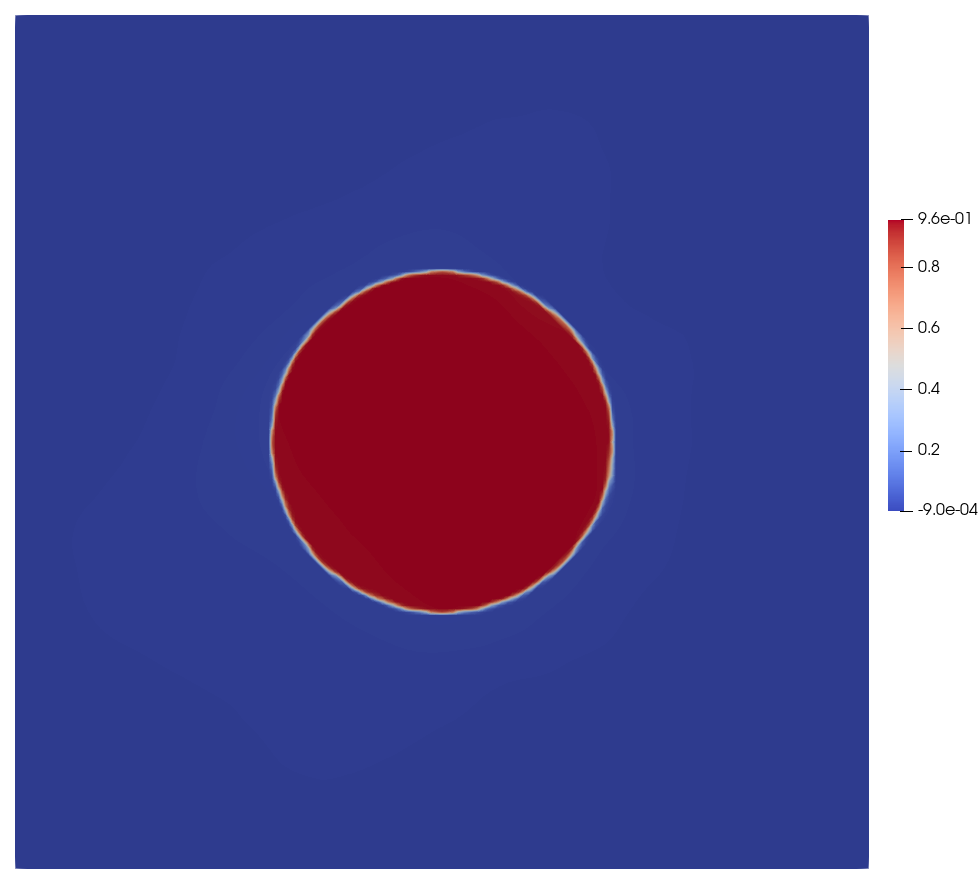}
\caption{Solution at time $t=0.0017,0.0038,0.05$ computed with tolerance $TOL_1$.}
\label{fig_sol_tol1}
\end{figure}
\begin{figure}[!htp]
\center
\includegraphics[width=0.24\textwidth]{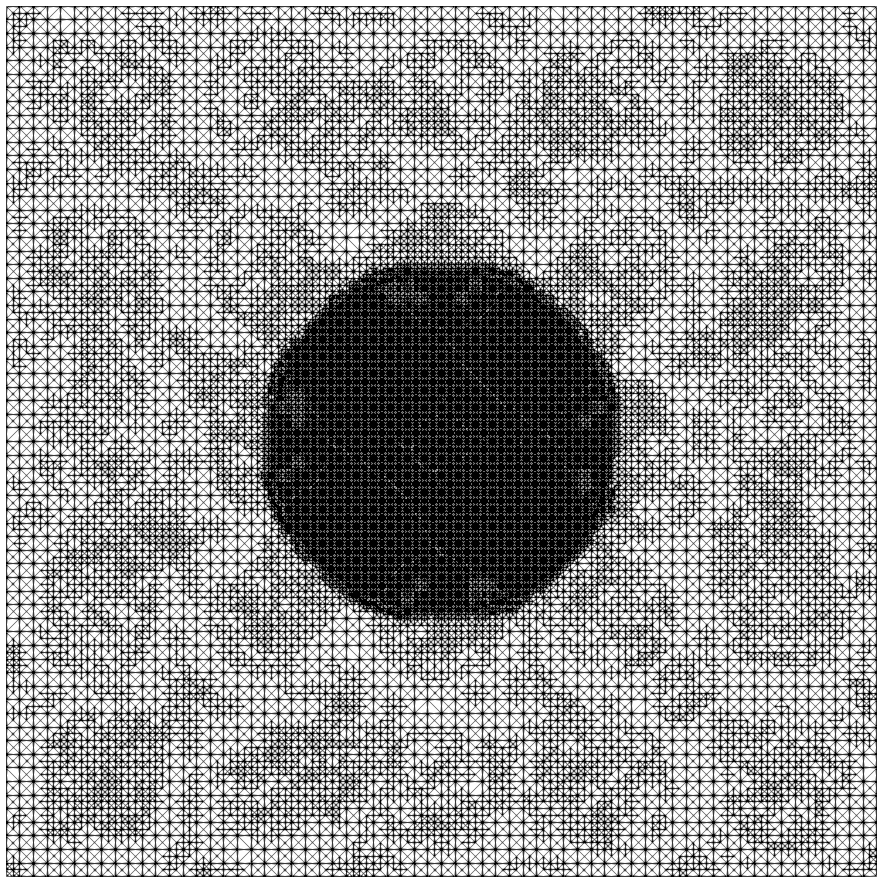}
\includegraphics[width=0.24\textwidth]{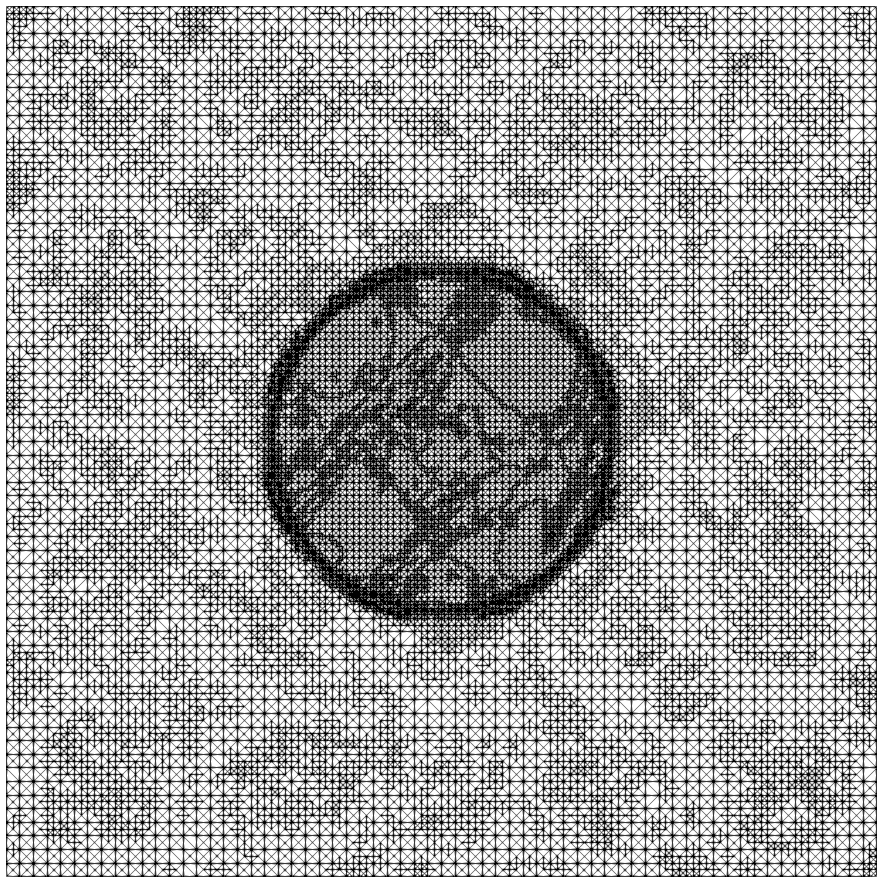}
\includegraphics[width=0.24\textwidth]{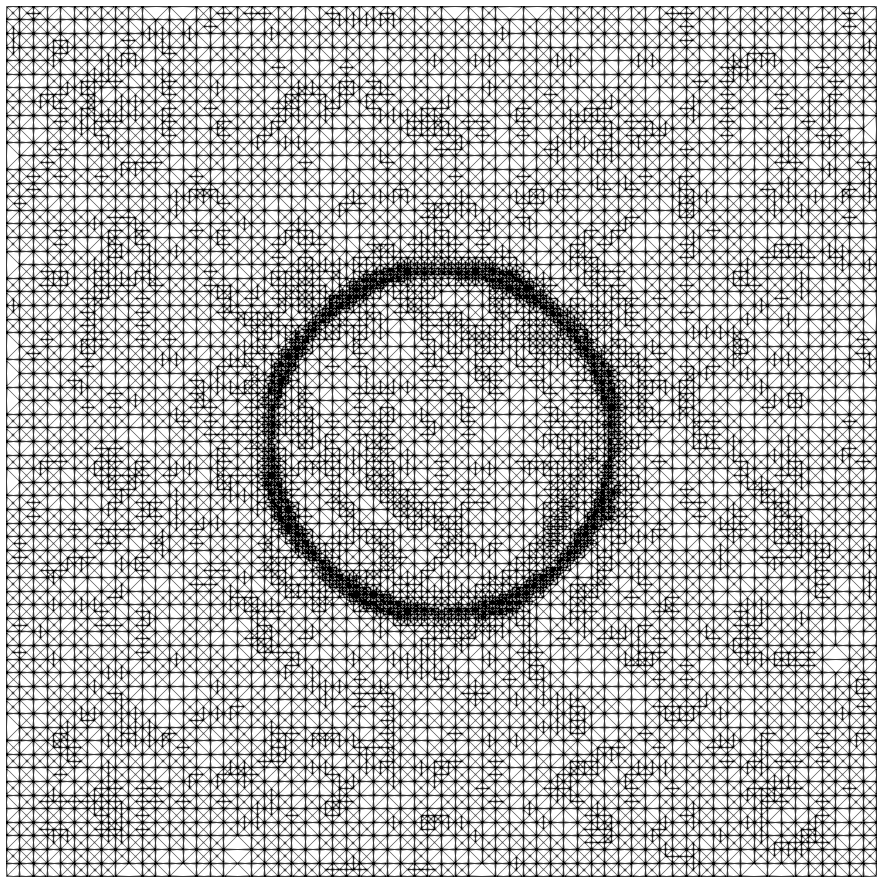}
\caption{Finite element mesh at time $t=0.0017,0.0038,0.05$ computed with tolerance $TOL_1$.}
\label{fig_mesh_tol1}
\end{figure}

Next, we study the behavior of the algorithm for the tolerance $TOL_k$, $k=0,1,2$.
In Figure~\ref{fig_fix_space} we display the evolution of the time-averaged spatial error indicator
and the corresponding number of degrees of freedom.
The time-averaged time error indicator and the corresponding time steps are displayed in Figure~\ref{fig_fix_time}.
For $t\approx T$ the numerical solution approaches a stationary state which is reflected by the growth of the time step towards the end of the computation.
\begin{figure}[!htp]
\center
\includegraphics[width=0.48\textwidth]{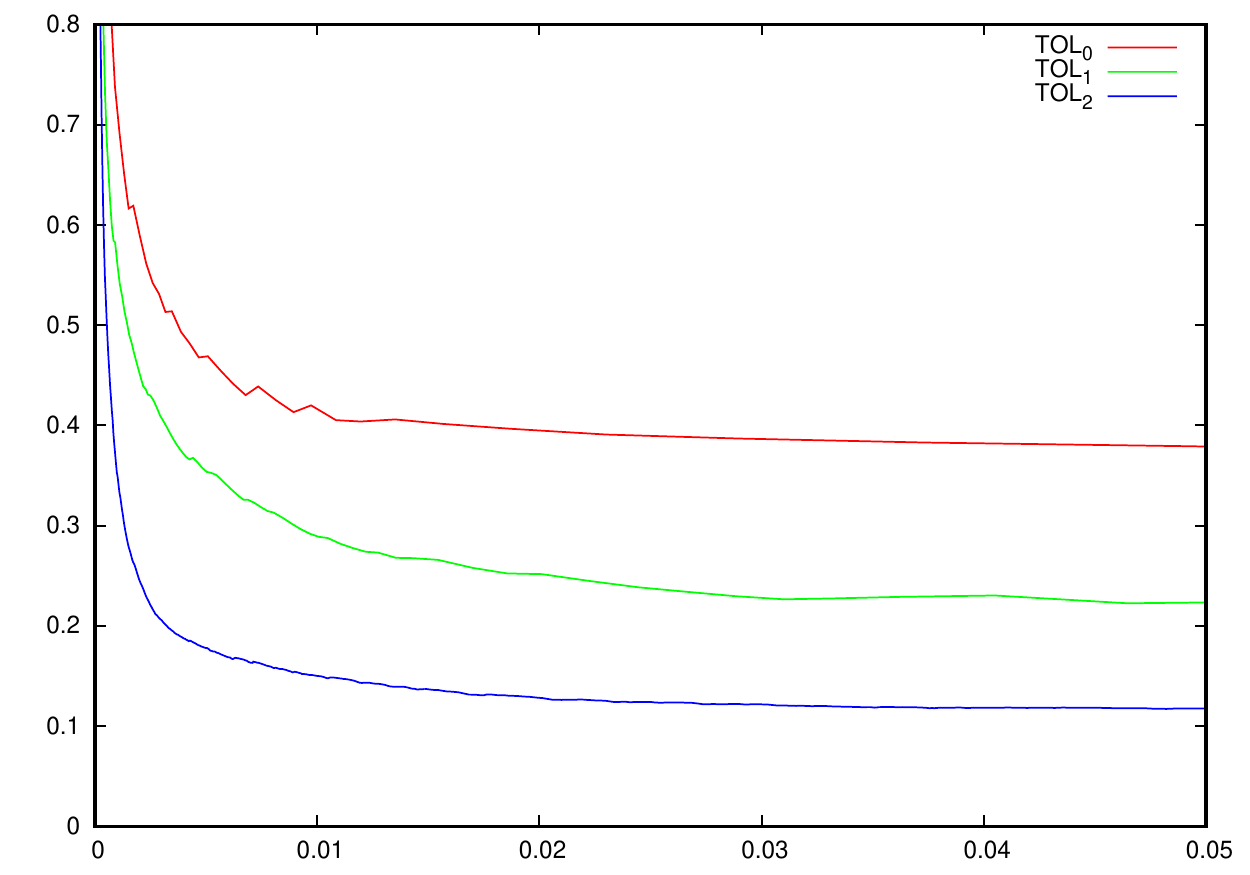}
\includegraphics[width=0.48\textwidth]{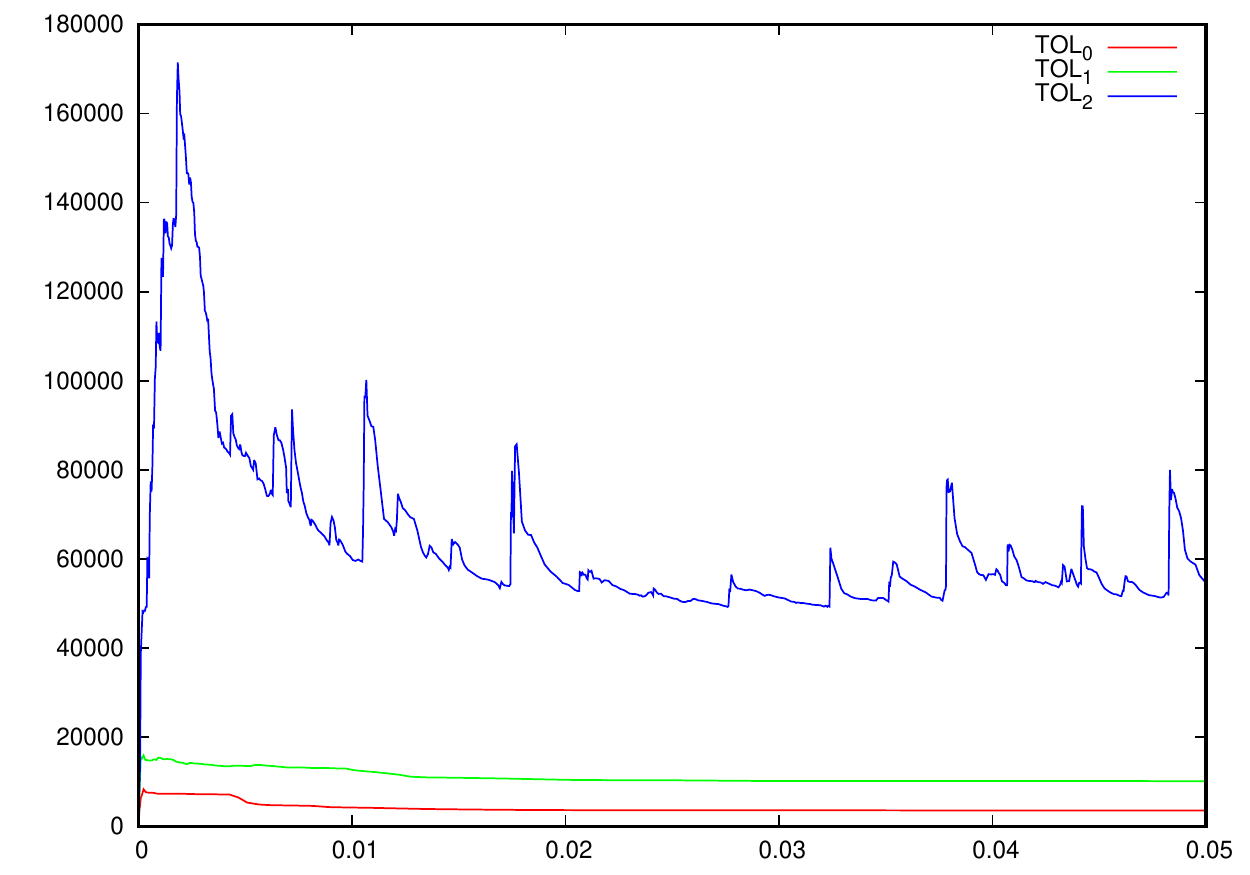}
\caption{Evolution of the spatial error indicator $t_n \rightarrow t_n^{-1} \sum_{i=1}^n \eta_{h}^i$ (left) and degrees of freedom $t_n \rightarrow \# \mathcal{T}_h^n$ (right) for $TOL_k$, $k=0,1,2$.}
\label{fig_fix_space}
\end{figure}
\begin{figure}[!htp]
\center
\includegraphics[width=0.48\textwidth]{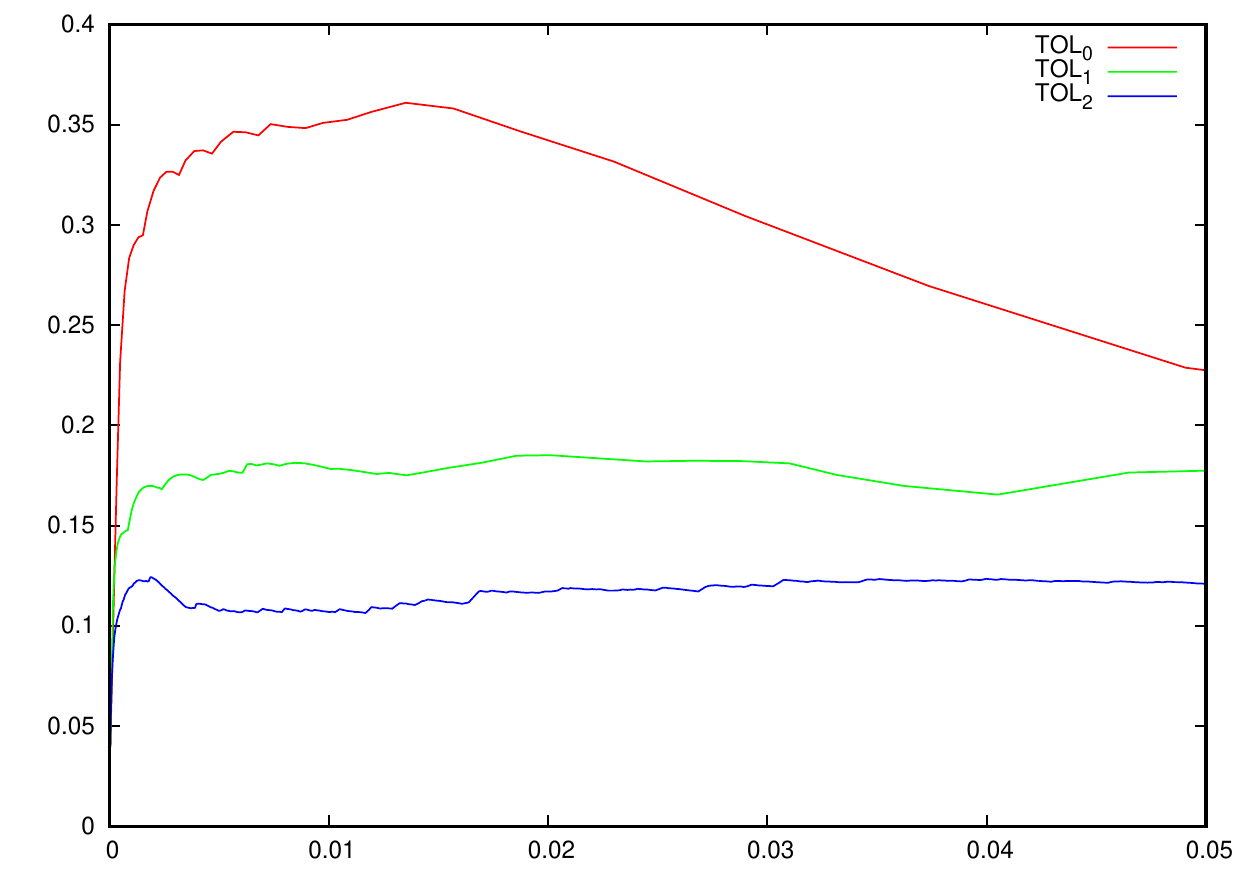}
\includegraphics[width=0.48\textwidth]{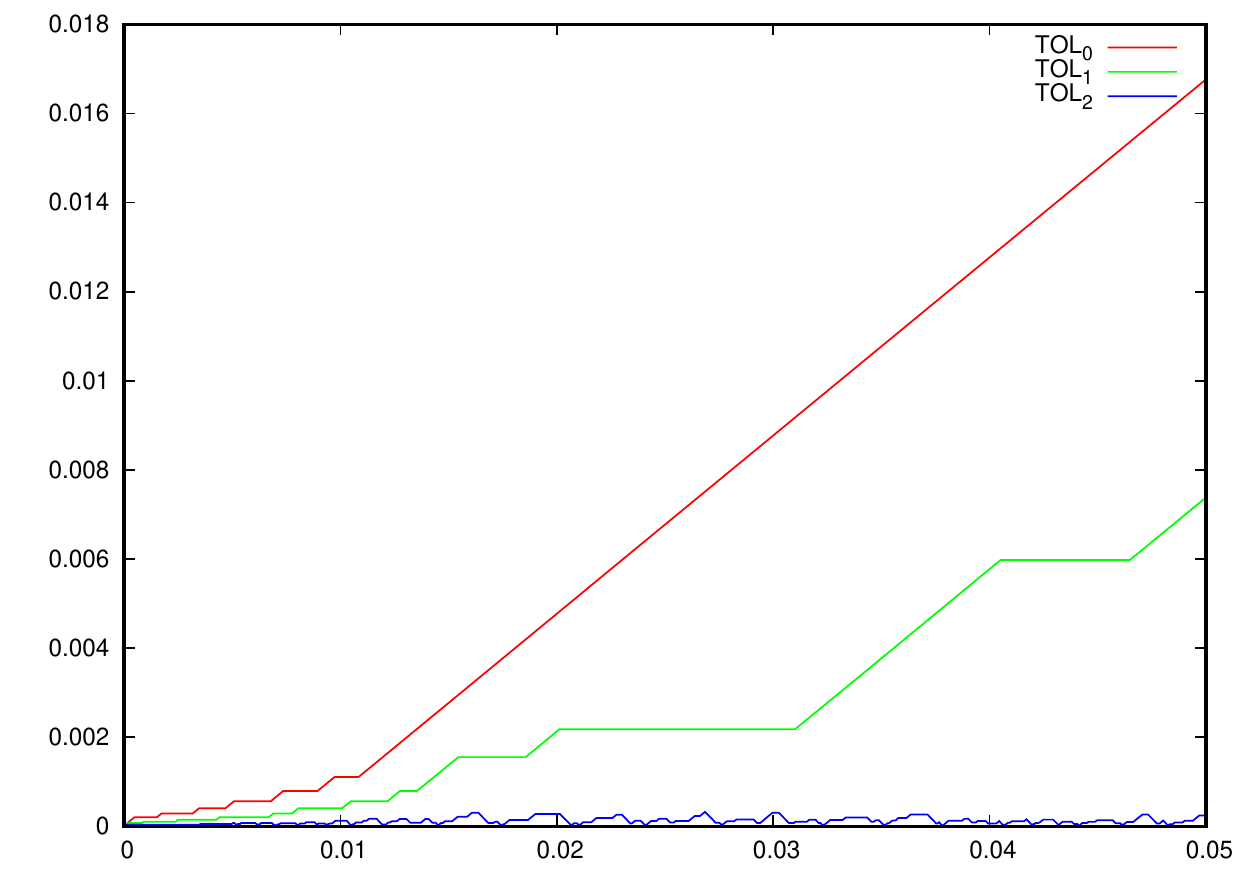}
\caption{Evolution of the time error indicator $t_n \rightarrow t_n^{-1} \sum_{i=1}^n \eta_{\mathrm{time},2}^i$ (left) and the timestep $t_n \rightarrow \tau_n$, $TOL_k$, $k=0,1,2$.}
\label{fig_fix_time}
\end{figure}

{\blue Next, we examine the effects of linearization 
on the solution.
We consider a (linear) semi-implicit scheme (SI), i.e. (\ref{fem_scheme})
with $|X^n_{\varepsilon,h}|_\eps$ in the the nonlinearity replaced by an explicit approximation $|X^{n-1}_{\varepsilon,h}|_\eps$.
Furthermore, we consider a linearized scheme obtained by the fixed point algorithm which is terminated after three iterations (FIX3)
as well as the original fixed point algorithm (FIX) with the tolerance $10^{-4}$.
In Figure~\ref{fig_lin} we display the time-averaged values of the linearization error indicator along with the time error indicator, the time step and the space error indicator.
We make the following observations. 
For the considered configuration of the parameters the linearization error indicator for the (SI) scheme attains similar values as the time error indicator, 
nevertheless for smaller tolerance (i.e., smaller time step and mesh size) the linearization error does not necessarily decrease.
The results obtained with the fixed point scheme (FIX3) indicate that the linearization error can be significantly reduced by performing just a few iterations.
Overall, the numerical results  obtained the three respective schemes (not displayed) were qualitatively very similar.

}

\begin{figure}[!htp]
\center
\includegraphics[width=0.48\textwidth]{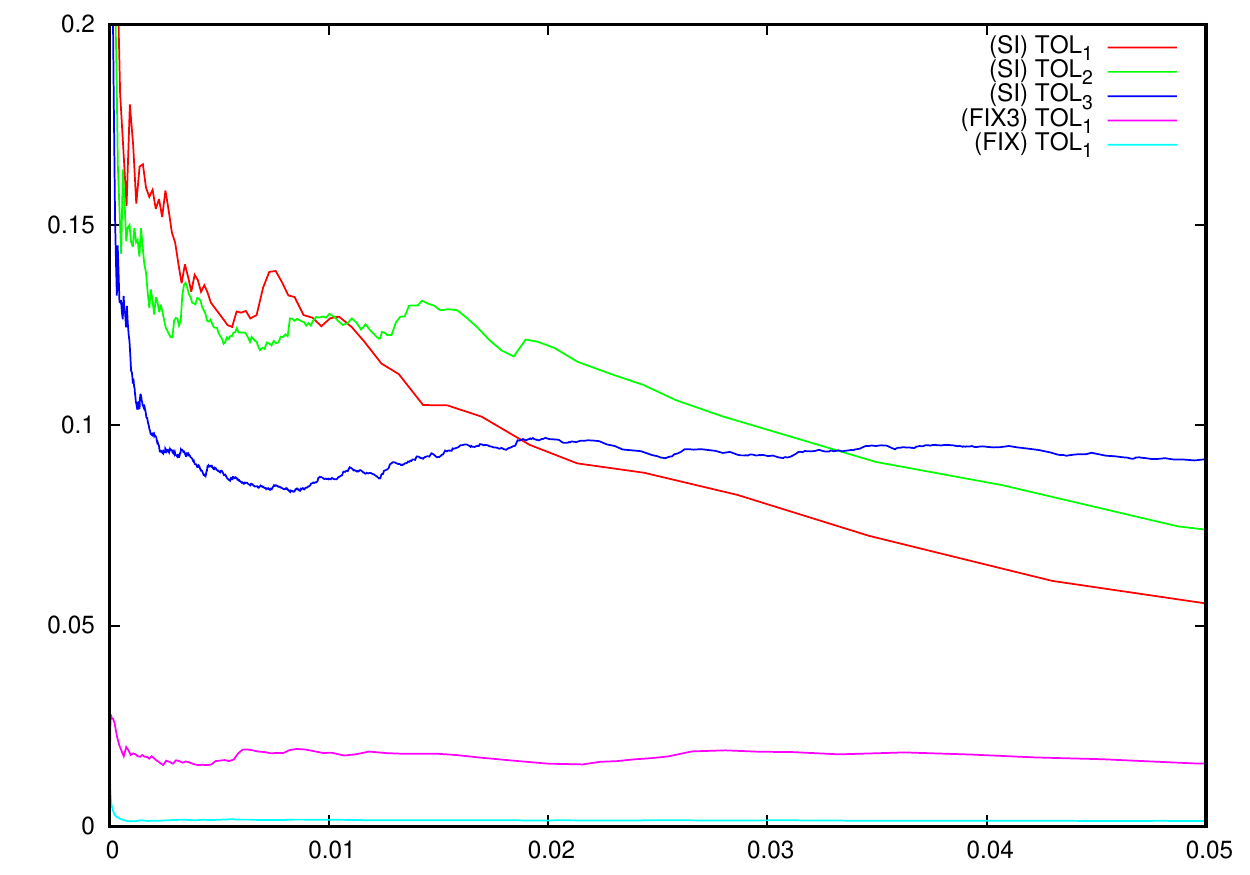}
\includegraphics[width=0.48\textwidth]{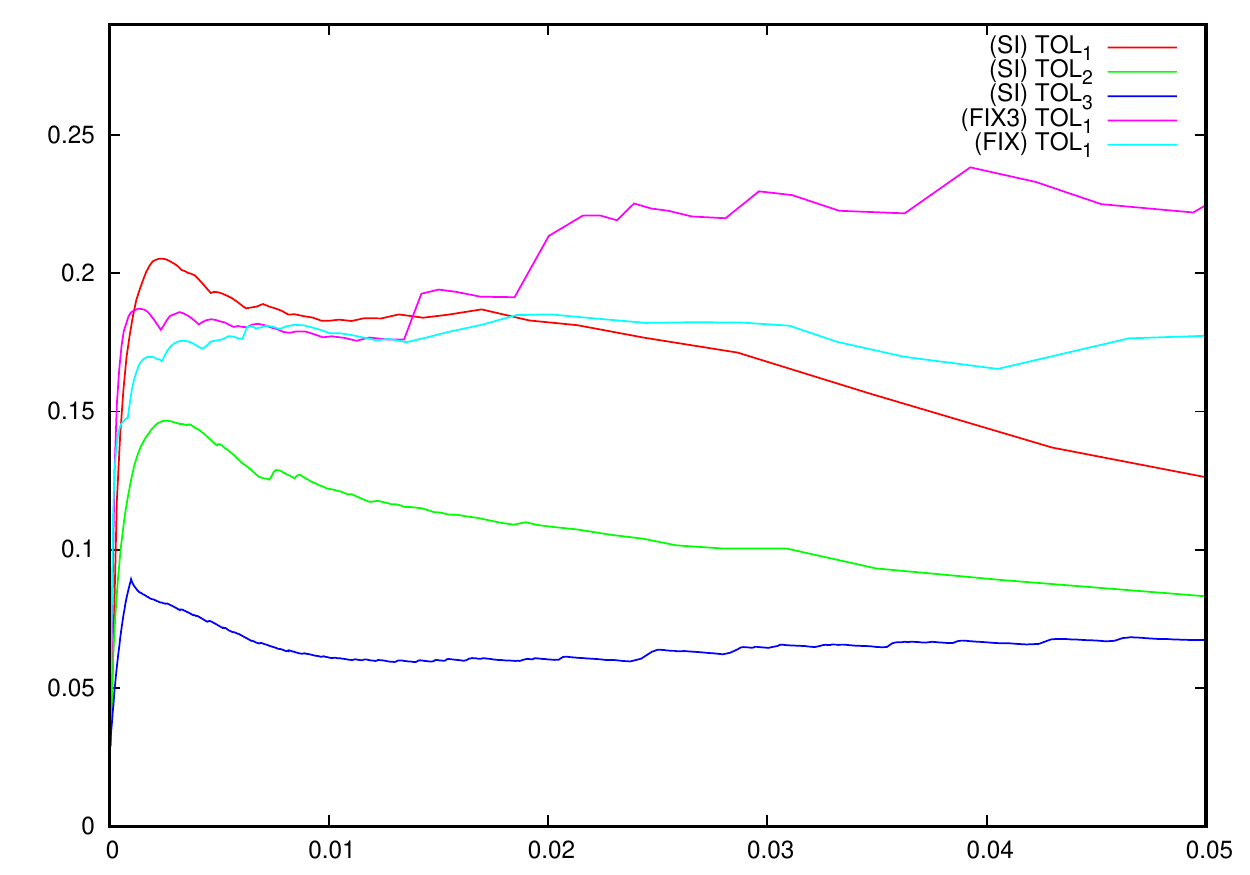}
\includegraphics[width=0.48\textwidth]{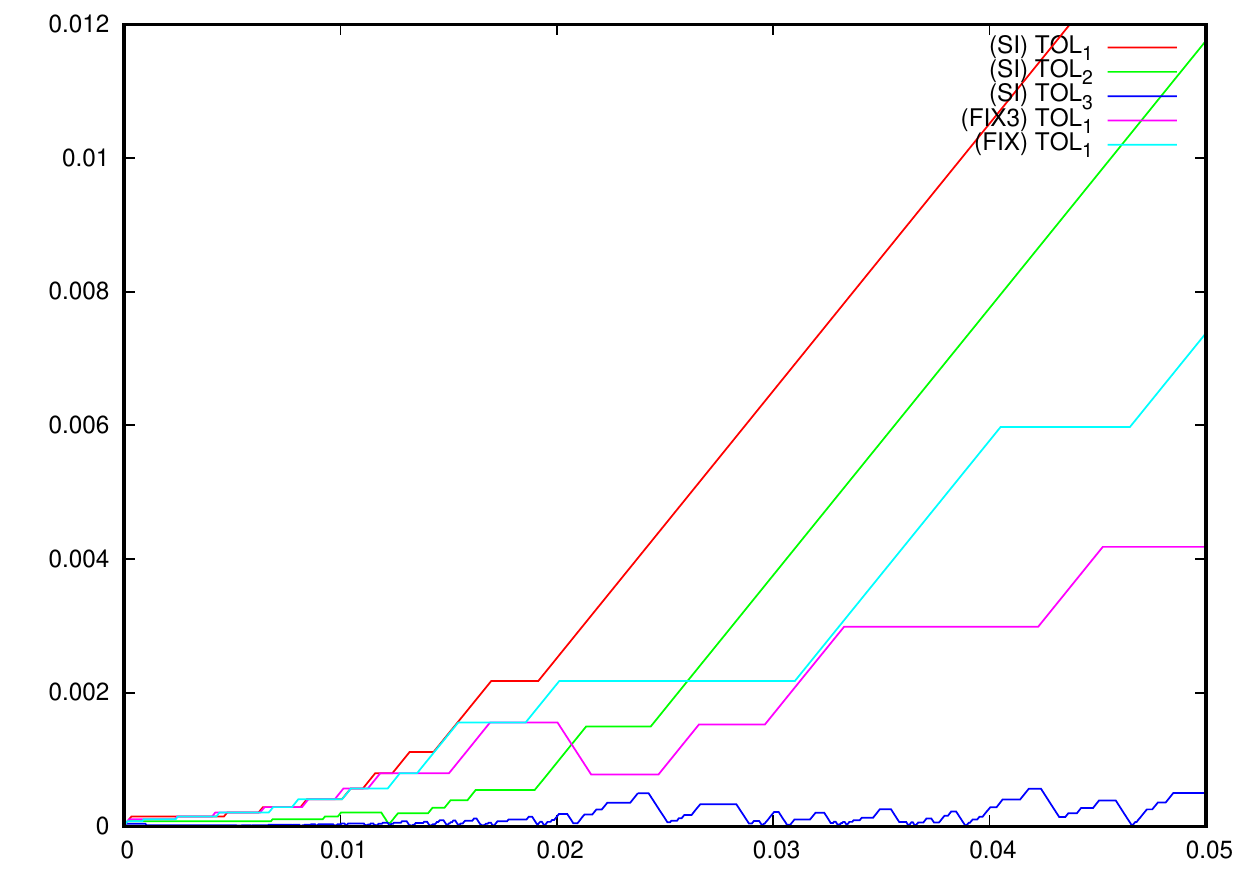}
\includegraphics[width=0.48\textwidth]{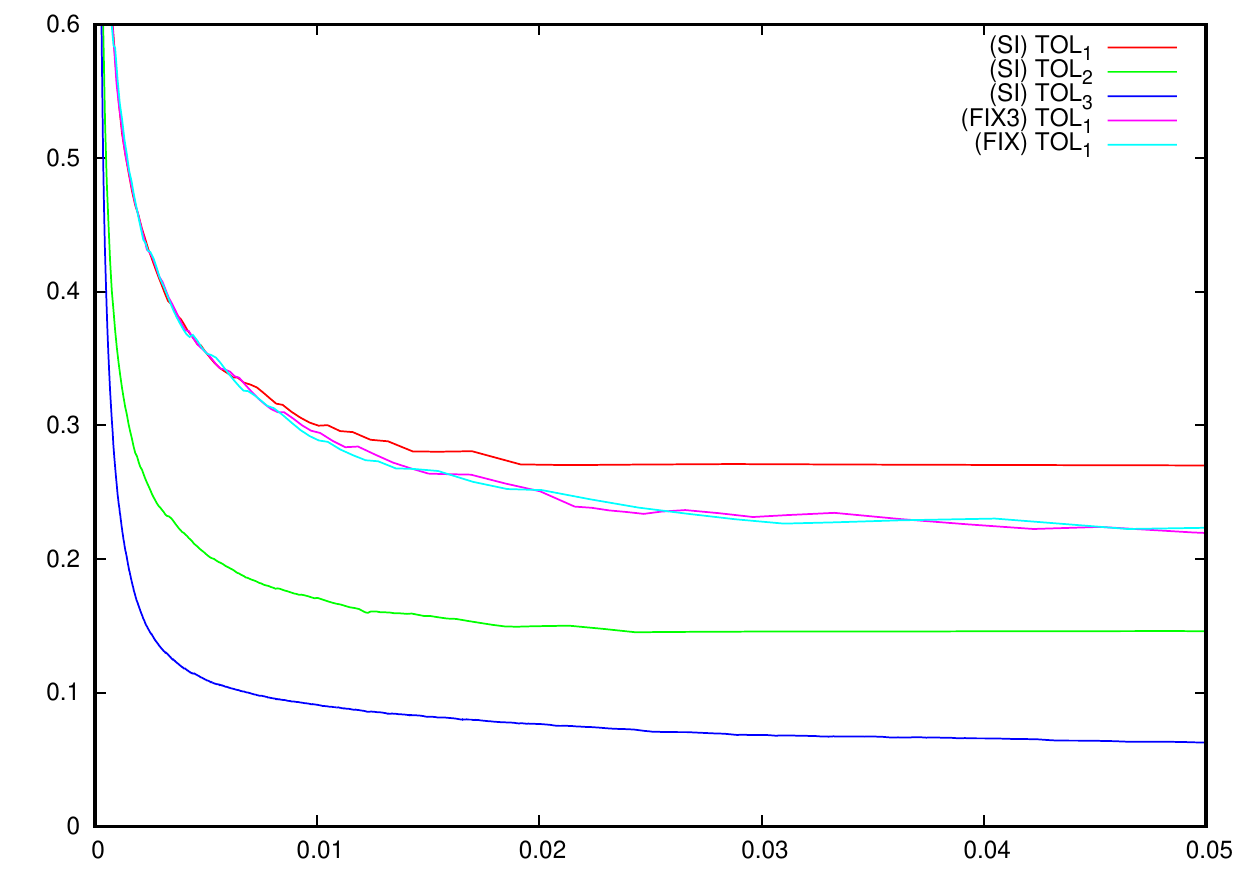}
\caption{Linearization error $t_n \rightarrow t_n^{-1} \sum_{i=1}^n \eta_{\mathrm{lin}}^i$ (top-left), time error $t_n \rightarrow t_n^{-1} \sum_{i=1}^n \eta_{\mathrm{time},2}^i$ (top-right),
timestep $t_n \rightarrow \tau_n$ (bottom-left), spatial error $t_n \rightarrow t_n^{-1} \sum_{i=1}^n \eta_{h}^i$ (bottom-right).}
\label{fig_lin}
\end{figure}

\section*{Acknowledgement}
We would like to thank the two referees careful reading of the paper and for their very inspiring comments and suggestions.


\bibliographystyle{plain}
\bibliography{refs_short}
\end{document}